\documentclass{MSI}
\usepackage{amsmath}
\usepackage{paralist}
\usepackage{graphicx}  \usepackage{epstopdf}%This is to transfer .eps figure to .pdf figure; please compile your paper using PDFLeTex or PDFTeXify.
 \usepackage[colorlinks=true]{hyperref}
\hypersetup{urlcolor=blue, citecolor=red}
\usepackage{centernot}
\def\currentvolume{X}
 \def\currentissue{X}
  \def\currentyear{200X}
   \def\currentmonth{XX}
    \def\ppages{X--XX}
     \def\DOI{10.3934/xx.xx.xx.xx}

\newtheorem{lemma}{Lemma}
\newtheorem{proposition}{Proposition}

\theoremstyle{definition}

\newtheorem{remark}{Remark}
\usepackage[version=4]{mhchem}

\newcommand\eps{\ensuremath\varepsilon}

\title[QSSA in open Michaelis--Menten system]{On the quasi-steady-state approximation in an open Michaelis--Menten reaction mechanism}

\author[J. Eilertsen, M. R. Roussel, S. Schnell and S. Walcher]{}


 \email{jueilert@umich.edu}
 \email{roussel@uleth.ca}
 \email{schnells@umich.edu}
 \email{walcher@matha.rwth-aachen.de}

\thanks{JE is supported by the University of Michigan Postdoctoral Pediatric Endocrinology and Diabetes Training Program ``Developmental Origins of Metabolic Disorder'' (NIH/NIDDK Grant: K12 DK071212). MRR's research is supported by the Natural Sciences and Engineering Research Council of Canada. The research of SW is supported by the bilateral project ANR-17-CE40-0036 and DFG-391322026 SYMBIONT}

\thanks{$^*$ Corresponding author: Sebastian Walcher ({\em E-mail address}: walcher@matha.rwth-aachen.de).}

\begin{document}
\maketitle

% Enter the first author's name and address:
\centerline{\scshape Justin Eilertsen}
\medskip
{\footnotesize
% please put the address of the first author
\centerline{Department of Molecular \& Integrative Physiology}
\centerline{University of Michigan Medical School}
\centerline{Ann Arbor, Michigan 49109, USA}
} % Do not forget to end the {\footnotesize by the sign }

\medskip

\centerline{\scshape Marc R. Roussel}
\medskip
{\footnotesize
 \centerline{Alberta RNA Research and Training Institute}
   \centerline{Department of Chemistry and Biochemistry}
   \centerline{University of Lethbridge}
   \centerline{Lethbridge, Alberta, Canada, T1K 3M4}
}

\medskip

\centerline{\scshape Santiago Schnell}
\medskip
{\footnotesize
\centerline{Department of Molecular \& Integrative Physiology}
\centerline{Department of Computational Medicine \& Bioinformatics}
\centerline{University of Michigan Medical School}
\centerline{Ann Arbor, Michigan 49109, USA}
}

\medskip

\centerline{\scshape Sebastian Walcher}
\medskip
{\footnotesize
\centerline{Mathematik A}
\centerline{RWTH Aachen}
\centerline{D-52056 Aachen, Germany}
}

\bigskip

% The name of the associate editor will be entered by an editorial staff
% "Communicated by the associate editor name" is not needed for special issue.
 \centerline{(Communicated by the associate editor name)}

%The abstract of your paper
\begin{abstract}
The conditions for the validity of the standard quasi-steady-state approximation in the Michaelis--Menten mechanism in a closed reaction vessel have been well studied, but much less so the conditions for the validity of this approximation for the system with substrate inflow. We analyze quasi-steady-state scenarios for the open system attributable to singular perturbations, as well as less restrictive conditions. For both settings we obtain distinguished invariant slow manifolds and time scale estimates, and we highlight the special role of singular perturbation parameters in higher order approximations of slow manifolds. We close the paper with a discussion of distinguished invariant manifolds in the global phase portrait.
\end{abstract}

\section{Introduction}\label{S:into}
Cellular function involves a large network of transformations of substrates, denoted~S, into products,~P, which in turn may be further transformed, eliminated, or cycled back into a useful form.
While the chemical conversion of S into P can occur spontaneously
\begin{align}\label{spon}
    \ce{S  ->[$k$]  P},
\end{align}
the rate constant, $k$, that regulates the speed of the reaction~(\ref{spon}) will often be very small, so that spontaneous conversion is too slow to sustain life. Moreover, spontaneous conversion allows only the crudest forms of control.  Consequently, the reaction must be \textit{catalyzed} or ``sped up.'' Enzymes, denoted~E, are biochemical catalysts that accelerate the conversion of S into P, and the chemical process by which the conversion of a substrate molecule into a product molecule is accelerated by an enzyme is called an \textit{enzymatic reaction}. 

The simplest description of an enzymatic reaction for a single-substrate, single-product reaction is the Michaelis--Menten mechanism~\cite{Henri1902,MM1913,Wurtz1880},
\begin{align}\label{mmcsd}
    \ce{S + E <=>[$k_1$][$k_{-1}$] C ->[$k_2$] E + P},
\end{align}
in which the conversion of S into P is achieved via two elementary reactions: the reversible formation of the enzyme-substrate complex,~C,  and the conversion of S to P in the complex~C with (in this simple model) simultaneous disassociation into E and~P. Enzymes lower the free-energy barrier separating reactants from products, with the result that (\ref{mmcsd}) is generally faster than~(\ref{spon}) by many orders of magnitude \cite[Section~6.2]{NC2008}.

The modeling and quantification of enzymatic reaction rates is of particular importance, especially since metabolic disease and dysfunction may arise when these reactions are too slow due, e.g., to a mutation in the corresponding gene. At or near the thermodynamic limit, enzymatic reactions are modeled by nonlinear ordinary differential equations (ODEs), known as rate equations, that obey the law of mass action. While the nonlinear terms in the model equations of enzymatic reactions make the mathematical treatment of the reaction mechanism challenging, avenues for simplification often exist. Specifically, if the rates of the elementary reactions that comprise the catalytic reaction are disproportionate, the ODE model will be \textit{multiscale}, meaning the complete reaction will consist of disparate slow and fast time scales. Under the influence of distinct fast and slow time scales, the rate of change of $c$ (using lower-case italic letters to represent the concentrations of the corresponding species) is very small relative to the rate of rate of change of~$s$. The exploitation of this almost negligible rate of change warrants a simplification of the form
\begin{align}\label{Qred}
    \ce{S  ->[$k_{\text{eff}}$]  P},
\end{align}
where $k_{\text{eff}}$ is the \textit{effective}---but non-elementary---rate function.
In the case of the Michaelis--Menten mechanism, $k_\text{eff}$ is a hyperbola in the variable~$s$; in more complicated mechanisms it may adopt the form of, for instance, a Hill-type function. The advantage offered by (\ref{Qred}) is that the entire reaction is describable in terms of the reactant concentration,~$s$, since the explicit dependence on $e$ and $c$ has been eliminated. The most widely studied example of this kind of reduction is probably the Michaelis--Menten rate law, also known as the standard quasi-steady-state approximation (sQSSA). More generally, rate laws of the form~(\ref{Qred}) are referred to as quasi-steady-state (QSS) reductions or quasi-steady-state approximations (QSSA). The term QSS speaks to the fact that the concentration of at least one chemical species (typically an intermediate) changes very slowly for the majority of the reaction. In fact, the rate of change is so small that it is \textit{nearly zero} (steady-state) but not quite; hence the expression \textit{quasi}-steady-state. 

The principal value of QSS approximations is that they yield a reduction of dimension \cite{2006-Flach-IPSB}.
In the biochemical arena, initially, the related equilibrium approximation was justified via biochemical arguments by Henri~\cite{Henri1902} and by Michaelis and Menten~\cite{MM1913}. Briggs and Haldane~\cite{BH1925} later provided a mathematical justification of the sQSSA using an argument that hints at later singular perturbation treatments but lacked formal justification. Only the development of singular perturbation theory some decades later (with seminal contributions by Tikhonov~\cite{Tikhonov1952}, and later Fenichel~\cite{Fenichel1979}) laid a solid mathematical foundation, which was used by Heineken et~al.~\cite{Heineken1967} to develop criteria for the validity of the sQSSA for the closed Michaelis--Menten system. This history was paralleled in inorganic chemistry, with the initial development of the sQSSA based on ad hoc chemical reasoning~\cite{Bodenstein1913,CU1913}, followed eventually by more rigorous treatments based on singular perturbation theory~\cite{BAO1963}.

Singular perturbation theory in this context applies to ODEs that depend on a small nonnegative parameter $\varepsilon$, and admit non-isolated stationary points at $\varepsilon=0$. In practice, e.g.\ for systems with polynomial or rational right-hand side, the set of stationary points then contains a submanifold of positive dimension, which is called a critical manifold. Given appropriate conditions (see Appendix \ref{AppendA} for details), one obtains a reduction to a system of smaller dimension constrained to evolve on the critical manifold. The challenge in any application of Fenichel theory resides in finding a small parameter from a given parameter dependent system. Traditional analyses of enzymatic reactions rely heavily on scaling and non-dimensionalization in order to transform the model equations into a standard form, and the utility of scaling analysis is that the small parameter often emerges naturally from the dimensionless equations \cite{Segel1989}. A different, more recent, approach~\cite{Goeke2015} starts with determining so-called Tikhonov--Fenichel parameter values (TFPV), by searching for parameter combinations at which the system admits non-isolated stationary points, and satisfies further technical conditions (see Section~\ref{TFPVsubs}).  From such (dimensional) TFPV one then obtains singular perturbation reductions via small perturbations along a curve in parameter space. In chemical applications, critical manifolds may emerge when specific system parameters (such as rate constants) vanish. 

While singular perturbation theory provides a very satisfactory toolbox for reduction of chemical reaction networks, examples from the literature indicate that the approach may be too narrow for some applications. Thus in some scenarios, at a certain parameter value there exists a distinguished invariant manifold which is, however, not comprised of stationary points.
 Formally, this means that a QSS reduction which approximates the system when $0<\eps\ll 1$ is not attributable to Fenichel theory. Nevertheless the QSS reduction is still sometimes a good approximation to the full system when Fenichel theory is inapplicable, and this raises several important questions. First, given the lack of a critical manifold and a fixed reduction procedure, how does one justify a QSS reduction, and how does one go about quantifying its efficacy? Second, if Fenichel theory is not applicable but a QSS reduction still proves to be an accurate approximation, will there be a distinguished invariant slow manifold that attracts nearby trajectories? In other words, what phase-space structures make the QSS reduction possible in situations where Fenichel theory is extraneous? In the present paper we contribute, on the one hand, to answering these questions for an open Michaelis--Menten system with constant substrate influx. On the other hand we provide sharper estimates for the accuracy of the sQSSA in singular perturbation scenarios. Finally, we consider distinguished invariant manifolds from a global perspective for the system on the Poincar\'e sphere.

\section{An open Michaelis--Menten reaction mechanism}

The open Michaelis--Menten reaction mechanism we consider here is the classical Michaelis--Menten reaction mechanism with a constant influx of substrate,~S, at a rate~$k_0$:
\begin{align}\label{mm1}
    \ce{->[$k_0$] S},\qquad\ce{S + E <=>[$k_1$][$k_{-1}$] C ->[$k_2$] E + P},
\end{align}
where $k_0$, $k_1$, $k_{-1}$ and $k_2$ are rate constants. 

Mathematical models for (\ref{mm1}) come in both deterministic and stochastic forms. Here we consider only the deterministic ODE model that follows the law of mass action near the thermodynamic limit. For a thorough analysis of the chemical master equation corresponding to (\ref{mm1}), we invite the reader to consult \cite{Othmer2020,Thomas2011}.

Let $s$, $e$, $c$ and $p$ denote the concentrations of S, E, C and P, respectively. The mass action model corresponding to~(\ref{mm1}) is given by the following set of nonlinear ODEs:
\begin{subequations}\label{MA}
\begin{align}
\dot{s} &= k_0 -k_1es + k_{-1}c,\label{sdot}\\
\dot{c} &= k_1es - (k_{-1}+k_2)c,\label{cdot}\\
\dot{e} &= -k_1es + (k_{-1}+k_2)c,\label{edot}\\
\dot{p} &= k_2c\label{pdot},
\end{align}
\end{subequations}
where ``$\dot{\phantom{x}}$'' denotes differentiation with respect to time. Summing equations~(\ref{cdot}) and~(\ref{edot}) reveals the conservation law
\begin{equation}\label{econ}
    c+e=e_T,
\end{equation}
where $e_T$ denotes the total enzyme concentration. Employing (\ref{econ}) to eliminate~(\ref{edot}), and noting that (\ref{pdot}) is not coupled to (\ref{sdot}) or~(\ref{cdot}), yields the simplified model
\begin{equation}\label{mmo}
\begin{array}{rcl}
\dot s&=& k_0-k_1(e_T-c)s + k_{-1}c,\\
\dot c&=& k_1(e_T-c)s -(k_{-1}+k_2)c,
\end{array}
\end{equation}
from which the the time dependence of $p$ and $e$ are readily obtained from (\ref{pdot}) and (\ref{econ}) once the solution to (\ref{mmo}) is procured. 

In contrast, the mass action system for the \textit{closed} Michaelis--Menten reaction mechanism is recovered by setting $k_0=0$:
\begin{subequations}\label{mmc}
\begin{align}
\dot s&= -k_1(e_T-c)s + k_{-1}c,\label{MMclosed_sdot}\\
\dot c&= k_1(e_T-c)s -(k_{-1}+k_2)c.\label{MMclosed_cdot}
\end{align}
\end{subequations}
One distinguishing difference between the open and closed system is that the total substrate concentration, $s_T$, is a conserved quantity when the reaction is closed. Therefore, (\ref{MA}) with $k_0=0$ is equipped with the additional conservation law $s_T=s+c+p$, whereas with $k_0>0$ one has only one conservation law,~(\ref{econ}). 

It is well known that further simplification of (\ref{mmc}) is possible via a QSS reduction. The most common reduction is the sQSSA, in which (\ref{mmc}) is approximated with a differential-algebraic equation consisting of the algebraic equation obtained by setting the right-hand side of equation~(\ref{MMclosed_cdot}) equal to zero (``$\dot{c}=0$'') along with the differential equation~(\ref{MMclosed_sdot}). This reduces to the single differential equation
\begin{subequations}\label{csQSSA}
\begin{align}
    \dot{s} &= - \cfrac{k_2e_Ts}{K_M+s}, \quad K_M :=\cfrac{k_{-1}+k_2}{k_1},\\
    c &= \cfrac{e_Ts}{K_M+s},\label{eq:c_of_s}
    \end{align}
\end{subequations}
where $K_M$ is the Michaelis constant.

The legitimacy of the sQSSA (\ref{csQSSA}) for the closed Michaelis--Menten reaction mechanism (\ref{mmc}) is well-understood. Following an early effort by Briggs and Haldane~\cite{BH1925}, Heineken, Tsuchiya, and Aris \cite{Heineken1967} were perhaps the first to prove with some degree of rigor that (\ref{csQSSA}) is valid provided $e_T \ll s_0$. The qualifier, $e_T\ll s_0$, was justified via singular perturbation analysis. Defining $\bar{s}:=s/s_0$, $\bar{c}:=c/e_T$, and $T:=k_1e_Tt$ generates the singularly perturbed dimensionless form of (\ref{mmc})
\begin{subequations}
\begin{align}
\bar{s}' &= -\bar{s}+\bar{c}(\bar{s}+\kappa-\lambda)\\
\mu\bar{c}'&= \bar{s}-\bar{c}(\bar{s}+\kappa),
\end{align}
\end{subequations}
where prime denotes differentiation with respect to $T$, $\lambda:=k_2/k_1s_0$, $\kappa:=K_M/s_0$, and $\mu:=e_T/s_0$. Consequently, the sQSSA (\ref{csQSSA}) is justified via Tikhonov's theorem~\cite{Tikhonov1952}. Throughout the years, refinements and variations of the condition $\mu\ll 1$ have been made. Perhaps most famously, Segel \cite{Segel1988} and Segel and Slemrod \cite{Segel1989} extended the results of Heineken et al.~\cite{Heineken1967} and demonstrated that (\ref{csQSSA}) is valid whenever $e_T\ll K_M +s_0$. Embedded in Segel's estimate is the more restrictive condition, $e_T\ll K_M$, which is independent of the initial substrate concentration, and is nowadays the almost universally accepted qualifier that justifies (\ref{csQSSA}) \cite{EILERTSEN2020}.

While the QSS reductions of the closed Michaelis--Menten reaction are well-studied, analyses pertaining to the validity of the QSSA in open reaction environments are somewhat sparse \cite{Othmer2020,GAO2011,Stoleriu2004,Thomas2011}. The question we address is therefore: when is further reduction of (\ref{mmo}) possible? The trajectories illustrated in Figure~\ref{fig:traj} show that there are certainly conditions under which the QSSA estimate of the  enzyme-substrate complex, given by equation~(\ref{eq:c_of_s}) which applies equally to the open system, is close to a slow invariant manifold, i.e.\ an invariant manifold (here, a trajectory) that attracts nearby trajectories and along which the equilibrium point is eventually approached from almost all initial conditions~\cite{Fraser1988,GK2003,RF91a}.
\begin{figure}
    \centering
    \includegraphics[scale=0.8]{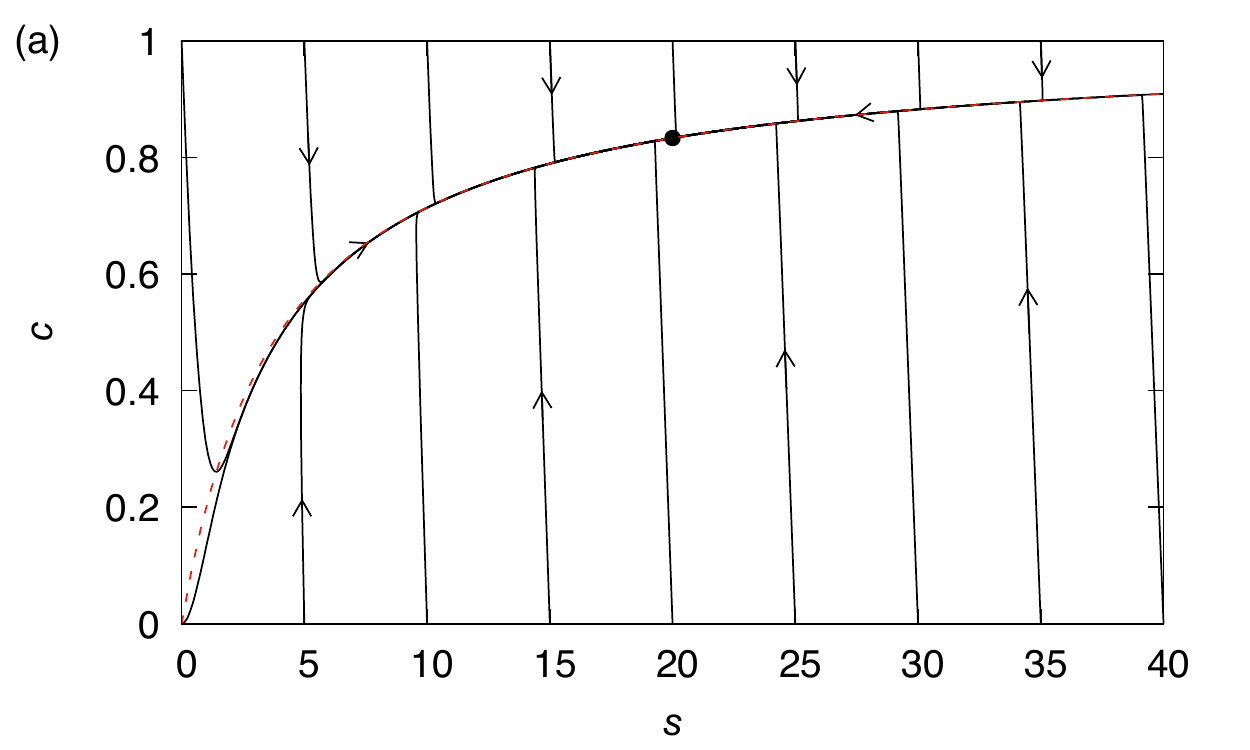}
    \includegraphics[scale=0.8]{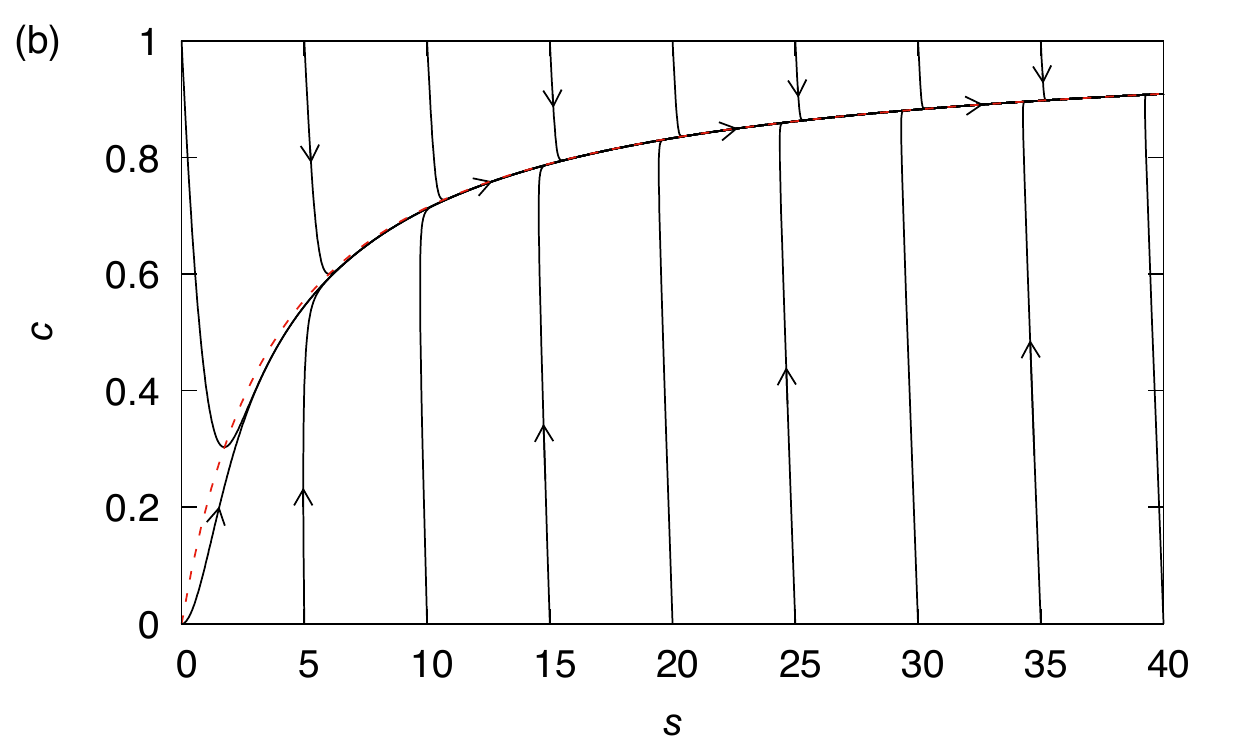}
    \caption{Trajectories of the open Michaelis--Menten equations~(\ref{mmo}) for (a) $k_1=1$,   $e_T=1$,   $k_{-1}=1$   $k_2=3$ and   
$k_0=2.5$ (in arbitrary units), i.e.\ under conditions where there is an equilibrium point in the first quadrant, marked by a dot; and (b) with parameters as in (a), except $k_0=3.5$, under which conditions there is not an equilibrium point in the first quadrant, and the $s$ component of the solution grows without bound.
The arrows show the direction of the flow. The dashed curve in both figures is defined by the QSSA equation~(\ref{eq:c_of_s}).
    \label{fig:traj}}
\end{figure}
We thus ask under what condition is the open sQSSA
\begin{equation}\label{osQSSA}
    \dot{s} = k_0 - \cfrac{k_2e_Ts}{K_M+s}, 
\end{equation}
permissible? At first glance, it seems rather intuitive to postulate that the open sQSSA~(\ref{osQSSA}) is valid under the same condition that legitimizes the closed sQSSA: $e_T \ll K_M$. In fact, following the earlier work of Segel and Slemrod \cite{Segel1989}, Stoleriu et~al.~\cite{Stoleriu2004} suggest that (\ref{csQSSA}) is applicable whenever
\begin{equation}\label{STOLcon}
e_T \ll s_0 + K_M\bigg(\cfrac{1}{1-\alpha}\bigg) + \cfrac{k_0}{k_2}, \quad \alpha:=k_0/k_2e_T,
\end{equation}
holds. The inequality (\ref{STOLcon}) is less restrictive than the Segel and Slemrod condition, since (\ref{STOLcon}) is satisfied as long as $k_0$ is sufficiently close to $k_2e_T$ (Implicitly, the authors assume that $\alpha<1$ in equation~\eqref{STOLcon}).

The approach used to derive (\ref{STOLcon}) was based on the traditional method of comparing time scales: a singular perturbation parameter was recovered through scaling analysis of the mass action equations (\ref{mmo}). However, it is possible to derive erroneous conclusions regarding the validity of the QSSA, even when great care is taken in scaling and non-dimensionalization methodology (see, for example \cite{Goeke2012}, Section 4). It thus seems prudent to reexamine the basis for the sQSSA in the open Michaelis--Menten mechanism using tools of singular perturbation theory that go beyond scaling arguments.

\section{The Quasi-Steady-State Approximation: Justification from singular perturbation theory}\label{GSPT}

In this section we derive the QSSA directly from Fenichel theory. Details covering projection onto the slow manifold can be found in Appendix~\ref{AppendA}.

\subsection{The critical manifolds: Tikhonov--Fenichel parameter values}\label{TFPVsubs}
To apply Fenichel theory to the open Michaelis--Menten reaction mechanism, we need a curve of non-isolated equilibrium solutions to form in the first quadrant of $\mathbb{R}^2$; see \cite{Goeke2015}. The following Lemma addresses the conditions that ensure the existence of a critical manifold, and records some general qualitative features.
\begin{lemma}\label{basicfacts}
\begin{enumerate}[(a)]
\item System \eqref{mmo} admits an infinite number of stationary points if and only if one of the following conditions holds.
\begin{itemize}
\item $k_0=k_1=0$;
\item $k_0=e_T=0$;
\item $k_0=k_2=0$.
\end{itemize}
\item If the number of stationary points in the plane is finite then it is equal to zero or one. There exists one stationary point if and only if the genericity conditions 
\begin{equation}\label{genercond}
k_1\not=0,\quad k_2\not=0 \text{  and    }k_2e_T-k_0\not=0
\end{equation}
are satisfied. In that case the stationary point is equal to
\begin{equation}\label{stapo}
P_0:=\left(\widehat s,\,\widehat c\right)=\left(\frac{(k_{-1}+k_2)k_0}{k_1(k_2e_T-k_0)},\frac{k_0}{k_2}\right).
\end{equation}
This point lies in the first quadrant if and only if
\begin{equation}
    k_2e_T-k_0>0,
\label{eq:enzcapcondition}
\end{equation}
in which case it is an attracting  node. The stationary point lies in the second quadrant if and only if $k_2e_T-k_0<0$,
in which case it is a  saddle point.
\item The first quadrant is positively invariant for system \eqref{mmo}, and solutions starting in the first quadrant exist for all $t\geq 0$. When $k_{-1}+k_2>0$ then every solution that starts in the first quadrant enters the (positively invariant) subset defined by $c\leq e_T$ at some positive time.
\item System \eqref{mmo} admits no nonconstant closed trajectory.
\end{enumerate}
\end{lemma}
\begin{proof}[Sketch of proof]
Parts (a) and (b) are straightforward, as is the first statement in part (c). For the second statement note $\dot s+\dot c\leq k_0$, hence solutions starting in the first quadrant remain in a compact set for all finite $t>0$. Finally, when $c\geq e_T$ then \eqref{mmo} shows that
$\dot c\leq-(k_{-1}+k_2)e_T$, hence the second statement of part (c) holds. We turn to the proof of part (d): If there exists a nonconstant closed trajectory then its interior contains a stationary point. Given a degenerate situation from part (a), the variety of stationary points is unbounded, hence would intersect a closed trajectory if it intersects its interior; a contradiction. This leaves the setting with an isolated stationary point, necessarily of index one, which is only possible when the stationary point \eqref{stapo} lies in the first quadrant. By part (c) the closed trajectory must be contained in the strip defined by $c\leq e_T$. But in this strip the divergence of the vector field equals $-\left( k_1(e_T-c)+k_1s+k_{-1}+k_2\right)<0$, and no closed nonconstant trajectory can exist by Bendixson's criterion.
\end{proof}
\begin{remark}
The case $k_0>k_2e_T$, in which the inflow exceeds the enzyme's clearance capacity, is not physiologically irrelevant since the gene coding for a particular enzyme may suffer a mutation that results in an enzyme with reduced catalytic activity, for example. As a rule, the accumulation of a metabolite will eventually become toxic (or possibly oncogenic) to the cell, and the rate at which S accumulates is therefore of interest. Other situations, e.g.\ the existence of an alternative but less efficient pathway for eliminating S, or the permeation of S through the cell membrane, would require more elaborate models for their study. Nevertheless, the model under study here would yield useful initial insights into the cellular effects of a mutation to an enzyme.
\end{remark}

Lemma (\ref{basicfacts}) ensures the existence of a critical manifold comprised of equilibrium points whenever $k_0$ vanishes along with either $e_T$, $k_1$ or $k_2$ in the singular limit. We note that in the context of the closed reaction (\ref{mmc}), parameters with $e_T=0$ (with all remaining parameters $>0$), respectively to $k_1=0$ (remaining parameters $>0$), respectively to $k_2=0$ (remaining parameters $>0$), are TFPV. Generally, a TFPV $[\widehat k_0\;\widehat e_T\;\widehat k_1\;\widehat k_2\;\widehat k_{-1}]$ is characterized by the property that a generic small perturbation results in the formation of a normally hyperbolic critical manifold \cite{Goeke2017}.

Let $\pi\in \mathbb{R}^5_+$ denote the parameter vector: $\pi :=[k_0 \;\; e_T\;\; k_1 \;\; k_2\;\;k_{-1}]^T$. The TFPVs and the critical manifolds, $M$, are as follows:
\begin{subequations}\label{Cman}
\begin{align}
\pi&=[0\;0\;k_1\;k_2\;k_{-1}] \implies M:=\{(s,c)\in \mathbb{R}^2: c=0\},\label{Cmank0eT}\\
\pi&=[0\;e_T\;0\;k_2\;k_{-1}] \implies M:=\{(s,c)\in \mathbb{R}^2: c=0\},\label{Cmank0k1}\\
\pi&=[0\;e_T\;k_1\;0\;k_{-1}] \implies M:=\{(s,c)\in \mathbb{R}^2: c=k_1e_Ts/(k_{-1}+k_1s)\}.\label{Cmank0k2}
\end{align}
\end{subequations}

Normal hyperbolicity and Fenichel theory ensure that perturbing $\pi$ in (\ref{Cman}) along a curve in parameter space through the TFPV results in the formation of an invariant slow manifold that attracts nearby trajectories at an exponential rate. Formally, the QSSA may be seen as an approximation of the dynamics on the slow manifold, perturbing from a TFPV.

\subsection{Singular perturbations and the geometry of parameter space}

The justification of the QSSA from singular perturbation theory requires us to implicitly equip parameter space with some additional geometric structure. For example, consider the case where both $e_T$ and $k_0$ vanish in the singular limit. In order to formally apply singular perturbation theory, it must hold that\footnote{The statement $k_0\sim O(e_T)$ is a bit awkward since $e_T$ and $k_0$ carry different units, but suitable dimensionless parameters will be discussed in Section \ref{GeneralQSS}.} $k_0 \sim O(e_T)$. Generally speaking, this means that we can apply singular perturbation theory along a parametric curve, $\Gamma$, in $(e_T,k_0)$ parameter space, $\Gamma:=(e_T,z(e_T))$, provided $z(0)=0$ and
\begin{equation}
\displaystyle \lim_{e_T\to 0^+} \cfrac{z(e_T)}{e_T} < \infty.
\label{z_over_eT}\end{equation}
However, a \textit{small} perturbation suggests that the parameter values will be \textit{close} to the parameter plane origin located at $(e_T,k_0)=(0,0)$. In this case $z(e_T)$ is well-approximated by its tangent line at $e_T=0$ (higher order terms in the Taylor expansion have no influence on the lowest order reduction), thus it is enough to only consider rays of the form $k_0=\gamma e_T$, where $\gamma$ is a positive constant with dimension $t^{-1}$. To eliminate the need for a dimensional slope $\gamma$, one can also consider rays of the form $k_0 = \alpha e_Tk_2$, where $\alpha$ is a dimensionless constant. Although, the easiest and perhaps clearest way to define a ray in parameter space is to set
\begin{equation}\label{ray}
    e_T\mapsto \eps e_T^*\quad\text{and}\quad k_0\mapsto \eps k_0^*,
\end{equation}
 where the parameters $k_0^*$ and $e_T^*$ are of unit magnitude and carry the units of $k_0$ and  $e_T$, respectively. 

The  additional constraint of sampling parameter space along a ray [or in a more general way along a curve satisfying~(\ref{z_over_eT})] must be imposed in order to justify the open sQSSA from singular perturbation theory. In their analysis of the open Michaelis--Menten reaction~(\ref{mmo}), Stoleriu et al.~\cite{Stoleriu2004} implicitly performed their analysis along a ray defined by
\begin{equation}\label{StolRay}
    e_T = e(0) + k_0/k_2,
\end{equation}
where $e(0)>0$ is the initial free enzyme concentration.   This ray in parameter space is encoded in their initial conditions, which allow for an arbitrary positive value of $e(0)$, but which specify $c(0)=\widehat{c}=k_0/k_2$. The advantage of working along the ray defined by~(\ref{StolRay}) is that there is no possibility that the inflow can exceed the clearance capacity of the enzyme, i.e.\ inequality~(\ref{eq:enzcapcondition}) is automatically satisfied.

In order to apply singular perturbation theory, we need to start from a critical manifold, i.e.\ from one of the cases in the set~(\ref{Cman}). Note that the ray through the $(e_T,k_0)$ parameter plane chosen by Stoleriu et al.~\cite{Stoleriu2004}, equation~(\ref{StolRay}), does not satisfy~(\ref{z_over_eT}) unless $e(0)=0$. This leads to difficulties. For example, the condition~(\ref{STOLcon}) \textit{along} the ray defined by~(\ref{StolRay}) translates to
\begin{equation}\label{STOLconRay}
k_1e(0)\ll k_1s_0 + (k_{-1}+k_2)\bigg(\cfrac{1}{1-\alpha}\bigg).
\end{equation}
The inequality (\ref{STOLconRay}) is satisfied by taking $k_1 \to 0$, but this limit alone does not produce a critical manifold. Hence, the singular perturbation  machinery is not obviously applicable to legitimizing the open sQSSA~(\ref{osQSSA}). 

Another issue with the constrained set of initial conditions imposed by~(\ref{StolRay}) is that it  excludes many initial conditions that are physiologically relevant. For example, a natural initial condition is $(s,e,c,p)(0)=(0,e_T,0,0)$, corresponding to the substrate flow being turned on at time zero (e.g.\ because the cell is placed in a new environment, or because it has turned on a previously dormant metabolic pathway that produces~S), but this initial point is \textit{inaccessible} if the parametric constraint~(\ref{StolRay}) has been imposed. Consequently, it remains an open question whether the results of the analysis apply at arbitrary points in parameter space and for arbitrary initial conditions. In particular, there is no guarantee that the analysis of Stoleriu et al.~\cite{Stoleriu2004} applies when the inflow exceeds the clearance capacity of the enzyme which, as argued previously, is not an irrelevant case. By contrast, a transformation informed by the basic requirements of singular perturbation theory such as~(\ref{ray}) allows us to make rigorous statements about the manifold structure of the problem, and imposes no constraints on the initial conditions.

\subsection{Quasi-steady-state reductions: Projecting onto the slow manifold}

Let us now consider the first scenario in which $e_T$ and $k_0$ vanish in the singular limit. The perturbation of the singular vector field is
\begin{equation}\label{p1}
\begin{array}{rcl}
\dot s&=& \eps k_0^*-k_1(\eps e_T^*-c)s + k_{-1}c,\\
\dot c&=& k_1(\eps e_T^*-c)s -( k_{-1}+k_2)c,\\
\end{array}
\end{equation}
Again, the singular limit obtained by setting $\eps=0$ in (\ref{p1}) yields a critical manifold, $M$, that is identically the $s$~axis:
\begin{equation}
    M:=\{(s,c)\in\mathbb{R}^2: c=0\}.
\end{equation}

To compute the corresponding singular perturbation reduction (see Appendix~\ref{AppendA} for specific details), we rewrite the right hand side of (\ref{p1}) as $P(s,c)f(s,c)+\eps G(s,c,\eps)$:
%SW: I made a distinction between singular perturbation reduction and QSS reduction; a priori they may not agree.
\begin{equation}
P(s,c) := \begin{bmatrix}k_1  s + k_{-1} \\
-k_1s -(k_{-1}+k_2)\end{bmatrix}, \quad f(s,c):=c, \quad G(s,c,\eps):=\begin{bmatrix}k_0-k_1e_Ts\\ k_1e_Ts\end{bmatrix}.
\end{equation}
The singular perturbation reduction is then obtained by projecting $G(s,c,0)$ onto the tangent space of $M$ at $x$ via the linear operator $\Pi^M$ which projects ``onto the kernel along the image'' of $N$:
\begin{equation}
 \Pi^M|_{c=0} G(s,0,0).
\end{equation}
For our specific problem (\ref{p1}), $\Pi^M$ is given by
\begin{equation}
\Pi^M := \begin{bmatrix} 1 & u(s)\\ 0 & 0\end{bmatrix}, \quad u(s):=\cfrac{(s+K_S)}{(s+K_M)}, \quad K_S:=k_{-1}/k_1,
\end{equation}
and the corresponding reduction, which agrees with the QSS reduction, is
\begin{equation}\label{QSSe}
  \dot{s}= k_0-\cfrac{k_2e_Ts}{K_M+s}.
\end{equation}
Equation~(\ref{QSSe}) is, of course, the open sQSSA. A similar calculation is easily carried out for the case of small $k_1$ and small $k_0$, as well as small $k_0$ and $k_2$, and we refer the reader to Appendix \ref{AppendA} for details. The specific QSS reduction that accompanies the perturbation defined by $k_1 \mapsto \eps k_1^*$ and $k_0 \mapsto \eps k_0^*$ is
\begin{equation}
    \dot{s} = k_0 - \cfrac{k_2e_T}{K_M}s,
\label{linlimit1}\end{equation}
which is the linear limiting law obtained in the small-$s$ limit of (\ref{QSSe}).

Accordingly, we have confirmation that the open sQSSA (\ref{osQSSA}) is valid under any condition that invokes a scaling of the form $k_0\mapsto \eps k_0^*$ and $e_T\mapsto \eps e_T^*$. We further note that a QSS reduction based on Fenichel theory is also possible in case $k_0 \mapsto \eps k_0^*$ and $k_2 \mapsto \eps k_2^*$ so that both $k_0$ and $k_2$ vanish in the singular limit. This reduction yields the classical equilibrium approximation (see, section~\ref{sec:k0k2manif} and Appendix~\ref{AppendA} for details).

Several questions remain. First, what is $\eps$? We have shown that the open sQSSA is valid provided $k_0$ and $e_T$ are sufficiently small, but \textit{what is small} when $k_0$ and $e_T$ are nonzero? Second, from the work of Goeke et al. \cite{Goeke2017}, the QSS may still hold in certain regions of the phase-plane even if Fenichel theory is not applicable. The analysis of Stoleriu et al. \cite{Stoleriu2004} is also indirectly suggestive of the idea that the validity the open sQSSA may not necessarily stem from singular perturbation theory. These observations raise the deeper question: is a scaling of the form $k_0 \mapsto \eps k_0^*, e_T\mapsto \eps e_T^*$ necessary for the validity of the QSSA, or merely sufficient? We address these questions directly in the sections that follow.

\section{Quasi-steady state for complex revisited}\label{GeneralQSS}
\subsection{The notion of QSS}
Singular perturbation theory provides a natural setting for developing conditions under which QSSA holds, but the literature (notably Stoleriu et al.~\cite{Stoleriu2004} for open Michaelis--Menten reaction mechanism) suggests that one should consider less restrictive notions as well. In the following we will sketch one such notion. This goes back to  Schauer and Heinrich \cite{Schauer1979}, who were the first to note that the minimal requirement for the validity of QSS reduction should be the \textit{near-invariance} of an appropriate QSS variety. The idea of near-invariance was expounded upon by Noethen et al.~\cite{Noethen2009}, and further analyzed by Goeke et al.~\cite{Goeke2017}: 
\begin{itemize}
\item As a starting point we take a fundamental feature of QSS for certain (sets of) species in a reaction: The rate of change for these species should be close to zero for an extended period of time. (In the Michaelis--Menten reaction, QSS for complex thus means that $\dot c\approx 0$ for an extended period of time.) In the phase space interpretation, a sizable part of the trajectory should thus be close to the {\em QSS variety} which is defined by setting the rates of change for the relevant species equal to zero (In the Michaelis--Menten reaction mechanism one thus has $c\approx k_1e_Ts/(k_1s+k_{-1}+k_2)$). The validity of such a condition will depend on the parameters.
\item  According to \cite{Goeke2017}, Section 3.3, the minimal requirement for QSS should therefore be near-invariance of the QSS variety, in the sense that the system parameters are small perturbations of QSS parameter values. By definition, at a QSS parameter value the QSS variety is an invariant set for system~\eqref{mmo}. (In the Michaelis--Menten reaction mechanism one thus has invariance of the variety defined by $c= k_1e_Ts/(k_1s+k_{-1}+k_2)$ for \eqref{mmo} at a QSS parameter value.) The arguments in  \cite{Goeke2017} show that this condition is necessary if one requires arbitrary accuracy of the QSS approximation for suitable parameters. By standard dependency theorems, small perturbations of a QSS parameter value yield trajectories that remain close to the QSS variety on compact time intervals; thus the condition is also sufficient.  One practical advantage of this notion is that QSS parameter values, similar to TFPV, are algorithmically accessible for polynomial or rational systems.
\item The near-invariance condition alone may not be considered sufficiently strong to satisfy expectations about QSS. One may also require that solutions quickly approach the QSS variety in an initial transient phase. Since the combination of these two features is automatically satisfied in singular perturbation settings, singular perturbations naturally enter the picture. But the singular perturbation scenario is both broader and narrower than QSS for chemical species: It is broader since it also is applicable to settings with slow and fast reactions. On the other hand, we will see below that it is, in a sense, too narrow for sQSS in the open Michaelis--Menten reaction mechanism.
\end{itemize}
\subsection{Open Michaelis--Menten: QSS parameter values for complex}
The QSS variety  for \eqref{mmo} is given by
\begin{equation}\label{qssvar}
    c=w(s):=\cfrac{k_1e_Ts}{k_1s+k_{-1}+k_2}.
\end{equation}
We prefer this to the usual notation
    $w(s)=e_Ts/(K_M+s)$, which may obscure the role of $k_1$. We first determine all QSS parameter values.
\begin{lemma}
The QSS parameters of system \eqref{mmo} are as follows:
\begin{enumerate}[(i)]
\item $e_T=0$ with the other parameters arbitrary;\label{QSSparameT}
\item  $k_1=0$ with the other parameters arbitrary;\label{QSSparamk1}
\item $k_0=k_2=0$;\label{QSSparamk0k2}
\item $k_{-1}=k_2=0$.\label{QSSparamkm1k2}
\end{enumerate}
\end{lemma}
\begin{proof} We proceed along the lines of \cite{Goeke2017}, Section 3.4, using an invariance criterion that employs the Lie derivative, $L[\cdot]$, corresponding to \eqref{mmo}. The Lie derivative is defined by
\[
\begin{aligned}
    L[\varphi](s,c)&=\dot{s}\cfrac{\partial \varphi}{\partial s} + \dot{c}\cfrac{\partial \varphi}{\partial c}\\
    &=\left(k_0-k_1(e_T-c)s+k_{-1}c\right)\cfrac{\partial \varphi}{\partial s} + \left(k_1(e_T-c)s-(k_{-1}+k_2)c\right)\cfrac{\partial \varphi}{\partial c}
\end{aligned}
\]
for any polynomial (more generally, smooth) function $\varphi$. For the variety defined by $\varphi=0$ to be invariant it is necessary that
\[
L[\varphi](s,c)=0 \text{  whenever }\varphi(s,c)=0.
\]
 Moreover, the condition is sufficient when $\varphi$ is irreducible, and it is applicable to the irreducible factors of $\varphi$; for details see \cite{Goeke2017} and the references therein. 

Now let $\psi(s,c)=0$ define the QSS manifold, thus
\begin{equation}
    \psi(s,c):=k_1(e_T-c)s-(k_{-1}+k_2)c.
\end{equation}
The invariance condition for the curve $\psi(s,c)=0$ is 
\begin{equation}
    L[\psi(s,c)]=-k_1(e_T-c)(-\psi(s,c)+k_0-k_2c)-(k_1s+k_{-1}+k_2)\psi(s,c)=0
\end{equation}
whenever $\psi(s,c)=0$, thus
\begin{equation}\label{LieF}
k_1(e_T-c)(k_0-k_2c)=0 \text{  whenever  } \psi(s,c)=0.
\end{equation}
This product yields three conditions which can be evaluated. Clearly $k_1=0$ works and yields~(\ref{QSSparamk1}). The second condition, $e_T-c=0$, holds on $\psi=0$ if and only if $(k_{-1}+k_2) e_T=0$, which yields (\ref{QSSparamkm1k2}) respectively to~(\ref{QSSparameT}). The third condition yields $k_0=k_2=0$ when $k_2=0$, i.e.~(\ref{QSSparamk0k2}). In case $k_2\not=0$ one obtains $c=k_0/k_2$, and
\[
k_1(e_T-k_0/k_2)s-(k_{-1}+k_2)k_0/k_2=0\text{  for all  }s;
\]
here the coefficient of $s$ and the constant must vanish. This again leads to conditions already discussed.
\end{proof}
\begin{remark}\label{qssprem}
\begin{enumerate}[(a)]
\item In cases (i) and (ii), the QSS variety is given by $c=0$, provided that the other parameters are positive, and the QSS parameter conditions are less restrictive than for singular perturbations, which also require $k_0=0$. This is a notable difference to the closed Michaelis--Menten scenario, for which all complex-QSS parameter values are also TFPV. Case (iii) corresponds to a singular perturbation scenario. The dynamics in case (iv) is of some interest in the Michaelis--Menten reaction mechanism without inflow; see \cite{EILERTSEN2020}.
\item Classical QSS reduction is tantamount to exploiting the fact that if $\psi(s,c)=0$ defines a nearly invariant curve, then $c\approx w(s)$, from which the open sQSSA~(\ref{osQSSA}) presumably follows. However, a word of caution is in order. When a QSS parameter value is also consistent with a singular perturbation and gives rise to a critical manifold, the classical QSS reduction may differ from the reduction obtained from Fenichel theory (see \cite{Goeke2017}, Section 3.5). For example, $\psi(s,c)=0$ is nearly invariant if $k_0$ and $k_2$ are small, but the classical QSS reduction, given by
\begin{equation}\label{falseQSS}
\dot{s} = k_0 - \cfrac{k_2e_Ts}{k_1s+k_{-1}},
\end{equation}
does not agree with the reduction obtained from singular perturbation theory, which is given by \eqref{nonqss}. 
Convergence to the singular perturbation reduction is guaranteed by Fenichel theory, hence the QSS reduction \eqref{falseQSS}  cannot correctly describe the dynamics at lowest order.
\end{enumerate}
\end{remark}
For the QSS parameters which do not correspond to singular perturbations, there remains to investigate whether solutions approach this variety, and if so, how fast and how close the approach is. Furthermore, even in the singular perturbation scenario one needs estimates on the initial (boundary layer) behavior, since Fenichel's theory applies directly only to a neighborhood of the critical variety. 

These problems will be addressed via direct estimates, which will also 
be of help in answering a quantitative question, i.e.\ how small should $e_T$ respectively to $k_1$ be in order to justify (\ref{osQSSA}). Ultimately, the term \textit{small} is relative in nature. Therefore, the appropriate question to ask is: For (\ref{osQSSA}) to be approximately accurate, $e_T$ and $k_0$ must be much smaller than \textit{what}? Before we start this investigation we establish an auxiliary result about the phase plane geometry of \eqref{mmo}.

\subsection{Phase plane arguments}\label{sec:pp}

From here on we restrict attention to system \eqref{mmo} on the positively invariant strip $W$ defined by $s\geq 0$ and $0\leq c\leq e_T$. A priori we impose no requirements on the parameters. We look at isoclines, noting that 
\begin{align}
&\dot c=0 \Leftrightarrow c=\mathcal{N}_c(s)\equiv\frac{k_1e_Ts}{k_1s+k_{-1}+k_2}, && \dot c\geq 0 \Leftrightarrow c\leq \mathcal{N}_c(s);\label{isoc}\\
\intertext{and} 
&\dot s=0 \Leftrightarrow c=\mathcal{N}_s(s)\equiv\frac{k_1e_Ts-k_0}{k_1s+k_{-1}}, && \dot s\geq0 \Leftrightarrow c\geq \mathcal{N}_s(s),\label{isos}
\end{align}
where $\mathcal{N}_x$ denotes the $x$~nullcline.
These nullclines define positively invariant sets:
\begin{lemma}\label{lemma:wedge} Consider the ``wedge''
\[
W_1:={\rm max}\,\left\{0,\,\frac{k_1e_Ts-k_0}{k_1s+k_{-1}}\right\}\leq c\leq \frac{k_1e_Ts}{k_1s+k_{-1}+k_2}, \quad s\geq 0.
\] Then the following hold:
\begin{enumerate}[(a)]
\item If the system admits no positive stationary point, thus $k_0>k_2e_T$, then the $c$-isocline lies above the $s$-isocline for all $s\geq 0$, and $W_1$ extends to $s\to\infty$. If the system admits the positive stationary point $(\widehat s,\widehat c)$ then the isoclines meet at this point, and $s\leq \widehat s$, $c\leq \widehat c$ for all points of $W_1$.
\item
$W_1$ is positively invariant for system \eqref{mmo}, and on $W_1$ one has $\dot s\geq 0$.
\end{enumerate}
\end{lemma}
\begin{proof} Part (a) is straightforward. As for part (b),
from \eqref{mmo} one sees that $\dot s +\dot c=k_0-k_2c\geq k_0-k_2\widehat c= 0$ on $W_1$, thus
\[
\dot s=0\Rightarrow \dot c\geq 0, \quad \dot c=0\Rightarrow \dot s\geq 0.
\]
This implies the positive invariance of $W_1$, since the vector field points to the interior of $W_1$ at the boundary (Figure~\ref{fig:wedge}). Clearly $\dot s\geq 0$ on $W_1$.
\end{proof}
\begin{figure}
    \centering
    \includegraphics[scale=0.9]{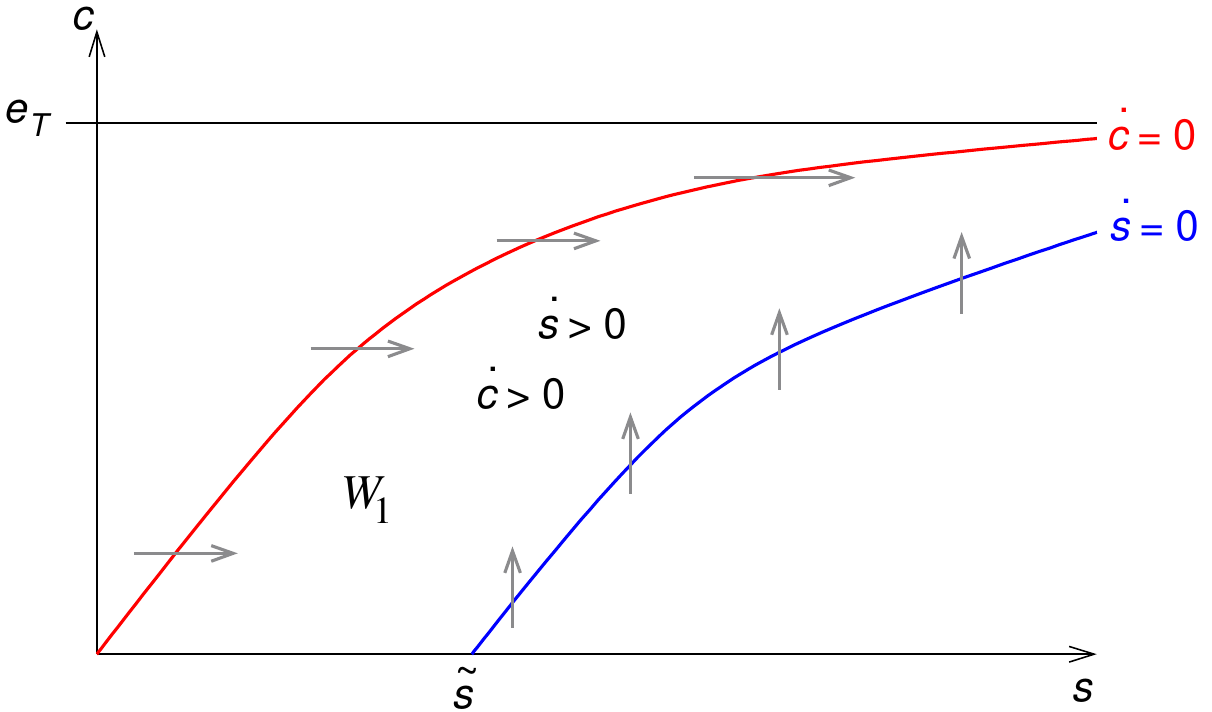}
    \includegraphics[scale=0.9]{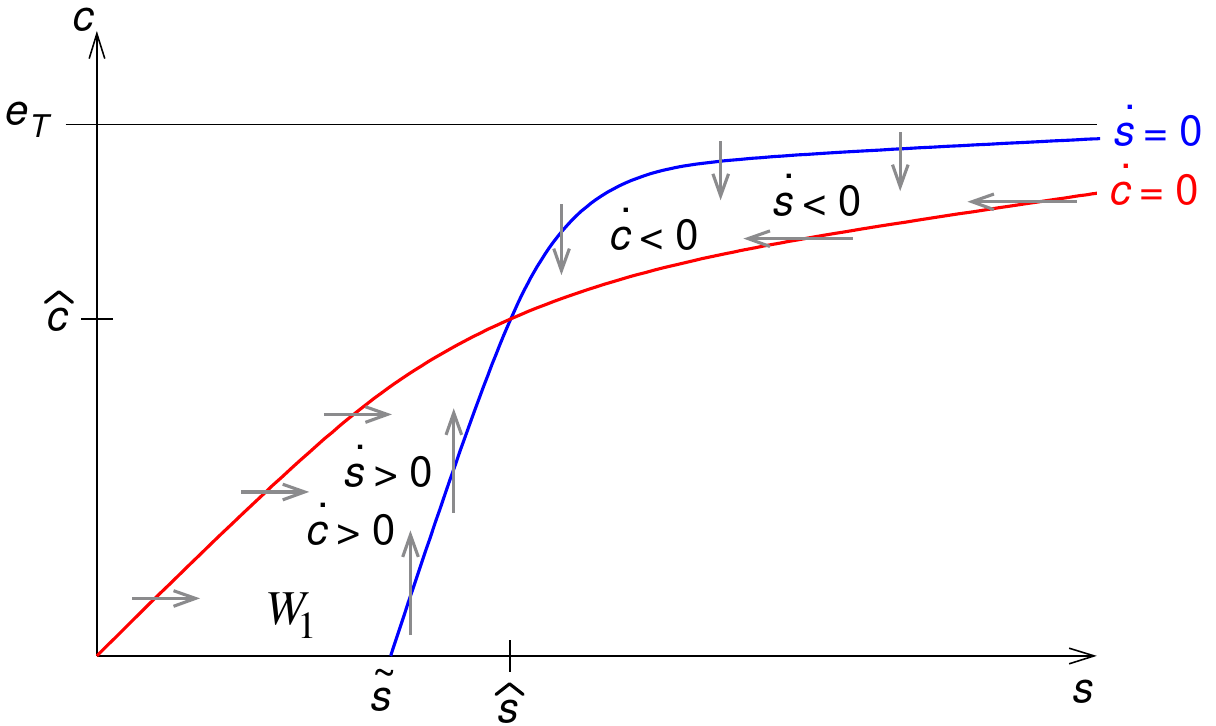}
    \caption{Sketches of the positively invariant sets $W_1$ in the phase plane for the open Michaelis--Menten reaction mechanism~(\ref{mm1}). The curves are the nullclines, and the arrows show the direction of motion of trajectories as they cross the nullclines. Both nullclines tend asymptotically to $c=e_T$ as $s\rightarrow\infty$. $\tilde{s}$ is the $s$~intercept of the $s$~nullcline. Upper: $k_0>k_2e_T$ and the two nullclines never meet. Lower: $k_2e_T>k_0$ and the nullclines cross at the stationary point $(\widehat{s},\widehat{c})$. The flow points into the region delimited by the two nullclines, making this region a funnel~\cite{HW1991}. 
    \label{fig:wedge}}
\end{figure}

\begin{remark}\label{RMK1}
Smallness of $e_T$ and existence of a positive stationary point imply smallness of $k_0$; this leads automatically to the singular perturbation setting. Matters are different when $k_1$ is small.
\end{remark}

Everywhere inside the wedge, $\dot{c}>0$ and $\dot{s}>0$. Thus, all trajectories inside the wedge have positive slope. Since the flow points into the wedge, the slow manifold must also lie inside the wedge. Thus, the slow manifold has a positive slope for $s\le\widehat{s}$ in the first quadrant. Moreover, the slow manifold must enter the first quadrant by crossing through the $s$~axis in the interval $(0,\tilde{s})$, where $\tilde{s}=k_0/k_1e_T$ is the $s$~intercept of the $s$~nullcline (Figure~\ref{fig:wedge}).

In the case that there is a positive equilibrium point, for $s>\widehat{s}$, $\dot{c}<0$ and $\dot{s}<0$ between the two nullclines so that trajectories in this region still have positive slope. The flow is, again,  into the region between the two nullclines (Figure~\ref{fig:wedge}), so the slow manifold must lie within this region. The slow manifold therefore has positive slope here as well. Moreover, $\lim_{s\rightarrow\infty}\mathcal{N}_c(s)=\lim_{s\rightarrow\infty}\mathcal{N}_s(s)=e_T$. Thus, the two nullclines pinch together asymptotically. Although we do not pursue this idea here, this property would allow the antifunnel theorem to be used to prove the existence of a unique slow manifold to the right of the equilibrium point~\cite{CS2008,HW1991} (see, Section~\ref{sec:poin} for correspondence to the global behavior).

\subsection{How small is \textit{small}: A direct estimate}\label{directEST}

Given that we are interested in obtaining a condition that ensures phase plane trajectories closely follow the QSS variety corresponding to the $c$-isocline (nullcline),  we compute an upper bound on the limit supremum ($\limsup$) of 
\[
L:=|c-w(s)|
\]
for a solution of \eqref{mmo}, where $w(s)$ is given by \eqref{qssvar}. To determine such an upper bound, we calculate
\begin{equation}\label{Eder}
    \cfrac{1}{2}\cfrac{d}{dt}L^2 =(c-h(s))(\dot{c}-w^\prime (s)\dot{s}).
\end{equation}
The derivative $\dot{c}$ given in (\ref{mmo}) factors nicely
\begin{equation}\label{fac}
\dot{c} = -k_1(s+K_M)(c-w(s))=:-\tau(s)(c-w(s)),
\end{equation}
and substitution of (\ref{fac}) into (\ref{Eder}) yields
\begin{subequations}
\begin{align}
\cfrac{1}{2}\cfrac{d}{dt}L^2 &=-\tau(s)L^2 - (c-h(s))(w^\prime(s)\dot{s})\\
&\leq -\tau_0L^2 + |L|\max{|w^\prime(s)}|\max|\dot{s}|, \quad \tau_0:=\tau(0).
\end{align}   
\end{subequations}
Differentiating $w(s)$ with respect to $s$ reveals $\max |w^\prime(s)|=k_1e_T/(k_{-1}+k_2)$. Denote $\max|\dot{s}|$ by $v$ and note that $v\leq k_0$ on $W_1$, due to $\dot s\geq 0$.

With
\[
\eps_c:=\cfrac{k_1e_T}{k_{-1}+k_2},
\]
Cauchy's inequality
\begin{equation}
    ab \leq \sigma a^2 + \cfrac{b^2}{4\sigma}, \quad \forall  \sigma >0
\end{equation}
implies
\begin{equation}
    \eps_c v |L| \leq \sigma L^2 + \cfrac{(\eps_c v)^2}{4\sigma}\quad \forall \sigma >0,
\end{equation}
which yields
\begin{equation}
  \cfrac{1}{2}\cfrac{d}{dt}L^2   \leq (\sigma-\tau_0)  L^2 + \cfrac{(\eps_c v)^2}{4\sigma}\quad \forall \sigma >0.
\end{equation}
A natural choice for $\sigma$ is $\sigma := \tau_0/2$ leading to the inequality
\begin{equation}\label{ineq}
    \cfrac{d}{dt}L^2   \leq -\tau_0 L^2 + \cfrac{(\eps_c v)^2}{\tau_0}.
\end{equation}
Applying Gronwall's lemma to (\ref{ineq}) generates an upper estimate for $L^2$:
\begin{proposition} 
\begin{enumerate}[(a)]
\item For every solution of \eqref{mmo} with initial value in $W_1$ one has the estimates
\begin{subequations}
\begin{align}
    L^2 &\leq L^2(0)e^{-\displaystyle \tau} + \cfrac{(\eps_c v)^2}{\tau_0^2}(1-e^{-\displaystyle \tau});\\
  L^2  &\leq L^2(0)e^{-\displaystyle \tau} + \cfrac{(\eps_c k_0)^2}{(k_{-1}+k_2)^2}
\end{align}
\end{subequations}
with $\tau:=\tau_0t=(k_{-1}+k_2)t$. 
\item Thus with
\begin{equation}
\varepsilon^*:=\cfrac{k_0k_1e_T}{(k_{-1}+k_2)^2},
\label{eq:epsilonstar}\end{equation}
the solution approaches the QSS variety up to an error of ${\varepsilon^*}^2$, with time constant $\tau_0=(k_{-1}+k_2)^{-1}$.
\end{enumerate}
\end{proposition}
Note that the estimates from the proposition explain the rapid approach of the trajectories in Figure \ref{fig:traj} to the QSS variety.

From our analysis of the mathematical energy, $L^2$, we have both a time constant, $\tau_0$, as well as a parameter, $\varepsilon^*$. The time constant is a natural \textit{dimensional} fast time scale, $\tau$, that is equivalent to the fast time scale obtained by Segel \cite{Segel1988} for the closed Michaelis--Menten reaction mechanism. Moreover, $\eps^*$ should in some sense be small for the open sQSSA to be accurate. The difficulty here is that $\eps^*$ has dimension, and we must scale $\eps^*$ appropriately to recover a dimensionless parameter. To scale, note that if $k_0<k_2e_T$, then
\begin{equation}\label{ineq2}
 \cfrac{e_T k_0}{K_M(k_{-1}+k_2)}  < \cfrac{k_2 e_T^2}{K_M(k_{-1}+k_2)}.
\end{equation}
Since $c \leq e_T$, we divide the (\ref{ineq2}) through by $e_T$, and take the inequality,
\begin{equation}\label{epsilon}
 \varepsilon_o:=\cfrac{k_2 e_T}{K_M(k_{-1}+k_2)} \ll 1, 
\end{equation}
to be the general qualifier for the validity of open sQSSA (\ref{osQSSA}) in $W_1$, when a finite stationary point is located in the first quadrant.

Note that $\eps_o$ vanishes if either $k_1$, $e_T$ or $k_2$ vanish. However, the use of Fenichel theory also requires $k_0$ to vanish in the singular limit, otherwise the perturbation is non-singular and the accuracy of a specific QSS reduction is attributable \textit{only} to the near-invariance of the QSS manifold (hence the difference in the justification of the open sQSSA that occurs from the mapping $(k_0,e_T)\mapsto \eps(k_0^*, e_T^*)$ versus the mapping $(k_0,e_T) \mapsto (k_0,\eps e_T^*)$). This observation is a definitive difference between our work and that of Stoleriu et al.~\cite{Stoleriu2004}.

%%%%%%%%%%%%%%%%%%%%%%%%%%%%%%%%%%%%%%%%%%%%%%%%%%%%%%%%%%%%%%%%%

\section{Additional insights from solutions of the invariance equation}\label{sec:invar}
The sQSSA is an attempt to approximate the slow invariant manifold. There are many other methods for approximating the slow manifold, ranging from the method of intrinsic low-dimensional manifolds~\cite{Maas1992}, which is accurate to $O(\varepsilon)$~\cite{KK2002}, to methods that can be improved order-by-order such as singular-perturbation theory~\cite{BAO1963,Heineken1967,Segel1989}, computational singular perturbation theory~\cite{Lam1993}, and Fraser's iterative method~\cite{Fraser1988,Nguyen1989}. Here, we study solutions of the invariance equation, the equation that the exact slow manifold satisfies, in order to gain further insights into the role of the TFPV in determining the validity of the sQSSA. The Fraser iterative method will be a major tool, but we will also consider various small-parameter expansions of the iterates.

\subsection{The invariance equation}
Assume that, in accordance with the arguments in Section~\ref{sec:pp}, and the visualization of Figure~\ref{fig:wedge}, we can represent the slow manifold (at least locally) as the graph of a function $c = \mathcal{C}(s)$. If $\dot{s}=\dot{s}(s,c)$ and $\dot{c}=\dot{c}(s,c)$, then differentiating the assumed representation of the slow manifold with respect to time, we get
\begin{equation}
    \dot{c}(s,\mathcal{C}) = \frac{d\mathcal{C}(s)}{ds}\,\dot{s}(s,\mathcal{C}),
\label{eq:invar0}\end{equation}
the invariance equation \cite{Fraser1988,Gorban2018,GKZ2004b,kuehn2015,Roberts89b}.

The invariance equation could be solved using a perturbation method. A strategy suggested by the work of the previous sections is to perturb from a TFPV along a curve in parameter space with the TFPV as its endpoint, e.g.~the ray~(\ref{ray}). The scaling parameter $\varepsilon$ can then serve as a perturbation parameter, and a perturbation problem of the typical form results, i.e.\ to compute the $i$'th term in the perturbation series, we solve an \emph{algebraic} equation that only depends on the previous terms. However, suppose that we did not know about TFPVs. Then we might try to use the same small parameter as in the closed system, viz.\ some scaled version of $e_T$ \cite{BH1925,Heineken1967,Segel1988,Stoleriu2004}. In the current framework, we would write $e_T\mapsto\varepsilon e_T^*$, and expand $\mathcal{C}(s)=\chi_1(s)\varepsilon + \chi_2(s)\varepsilon^2 +\ldots$ If we implement this program, we find that $\chi_1(s)$ satisfies the \emph{differential} equation
\begin{equation}
    \frac{d\chi_1}{ds} = \frac{1}{k_0}\left[k_1e_T^*s-\chi_1(k_1s+k_{-1}+k_2)\right].
\end{equation}
Higher-order terms also satisfy differential rather than algebraic equations. These difficulties are linked to the fact that $e_T$, of itself, is not a TFPV for the open system. For the TFPVs (\ref{Cmank0eT}) and~(\ref{Cmank0k1}), since the leading-order term in $\mathcal{C}(s)$ is~$O(\varepsilon)$, rescaling the TFPVs balances the terms in the invariance equation such that, to leading order, $\dot{c}$, $\dot{s}$ and $d\mathcal{C}/ds$ are all~$O(\varepsilon)$. As a result, (with slight abuse of notation) $d\chi_i/ds$ first appears to~$O(\varepsilon^{i+1})$, and we obtain an algebraic equation for~$\chi_i$. The case of TFPV~(\ref{Cmank0k2}) is slightly different. If we rescale $(k_0,k_2)\mapsto\varepsilon(k_0^*,k_2^*)$ and take $\mathcal{C}(s)=\zeta_0(s)+\zeta_1(s)\varepsilon+\zeta_2(s)\varepsilon+\ldots$, the $\varepsilon^0$ terms of the invariance equation can be rearranged to
\begin{equation}
    \left[k_1s(e_T-\zeta_0)-k_{-1}\zeta_0\right]\left(1+\frac{d\zeta_0}{ds}\right) = 0.
\label{zeta0eq}\end{equation}
The term in square brackets gives us the critical manifold~(\ref{Cmank0k2}) for $\zeta_0$ (the other solution, $d\zeta_0/ds = -1$, gives the fast foliations of the manifold in the limit $\varepsilon\rightarrow 0$). At higher orders, $\zeta_i$ first appears with the $O(\varepsilon^i)$ terms. However, because in the limit $\varepsilon\rightarrow 0$ for this TFPV set, the $k_0$ term in $\dot{s}$ vanishes, the coefficient of $d\zeta_i/ds$ at  $O(\varepsilon^i)$ is the term in square brackets in equation~(\ref{zeta0eq}), which vanishes. Thus, $d\zeta_i/ds$ first appears with a non-vanishing coefficient at~$O(\varepsilon^{i+1})$, and we again have a perturbation problem involving only algebraic equations.

To recapitulate, rescaling the TFPVsyields a tractable perturbation problem precisely because the TFPVs define critical manifolds. Choosing any  path through parameter space that does not reduce to a TFPV as $\varepsilon\rightarrow 0$ will, by contrast, necessarily yield a troublesome perturbation problem.

Each TFPV set yields a different perturbation problem. Rather than studying the perturbation expansions of the slow manifold directly, we turn to Fraser's method~\cite{Fraser1988,Nguyen1989}, which will allow us to compute a sequence of approximations in a TFPV-agnostic manner. Series expansions of the approximations can then be obtained for any desired TFPV scaling parameter.

In Fraser's iterative method, we think of the invariance equation as an equation to be solved for $\mathcal{C}$ in terms of $d\mathcal{C}/ds$.
In this case, we can explicitly rearrange the invariance equation to the functional equation~\cite{Fraser1988}
\begin{equation}
    \mathcal{C} = \frac{\displaystyle k_1e_Ts\left(1+\frac{d\mathcal{C}}{ds}\right) - k_0\frac{d\mathcal{C}}{ds}}{(k_1s+k_{-1})\left(1+\displaystyle\frac{d\mathcal{C}}{ds}\right) + k_2}.
\label{eq:invar}\end{equation}
Observe that if we rescale $k_0$ and $e_T$ as in~(\ref{ray}) and let $\varepsilon\rightarrow 0$ in the functional equation, we recover the critical manifold~(\ref{Cmank0eT}). Similar comments can be made for $(k_0,k_1)$ and $(k_0,k_2)$ and the corresponding critical manifolds (\ref{Cmank0k1}) and~(\ref{Cmank0k2}), respectively. Thus, the critical manifolds are recovered in suitable limits of the functional equation. This reinforces the special relationship of the TFPVs to perturbative solutions of the invariance equation.

We now want to solve equation~(\ref{eq:invar}) in some way that constrains the calculation to represent the slow manifold, which is potentially an issue because every trajectory that can locally be represented in the form $c=\mathcal{C}(s)$ is a solution of the invariance equation. If we knew the derivative of $\mathcal{C}$ with respect to~$s$ along the slow manifold, we could immediately compute $\mathcal{C}(s)$ from~(\ref{eq:invar}). Since we do not, we solve the invariance equation by iteration: From some initial guess $\mathcal{C}_0(s)$, we compute the derivative, substitute it into~(\ref{eq:invar}) to obtain $\mathcal{C}_1(s)$, and iterate.
Despite the potential to find another trajectory by this procedure, in practice, we find that iterative solution of a functional equation such as~(\ref{eq:invar}) tends to converge specifically to the slow manifold~\cite{Fraser1988,Roussel:1990:GSSA} if it converges at all~\cite{Roussel97}.

The critical manifolds associated with the TFPVs suggest potential initial functions for iteration. Suppose then that we start iteration from the critical manifold [under either TFPV (\ref{Cmank0eT}) or~(\ref{Cmank0k1})] $\mathcal{C}_0(s) = 0$. Then $\mathcal{C}_1(s)$ is the sQSSA~(\ref{eq:c_of_s}). Figure~\ref{fig:iterates} shows a sequence of iterates calculated from this initial function. Convergence is rapid, although much more so away from the $s$~axis.
\begin{figure}
    \centering
    \includegraphics{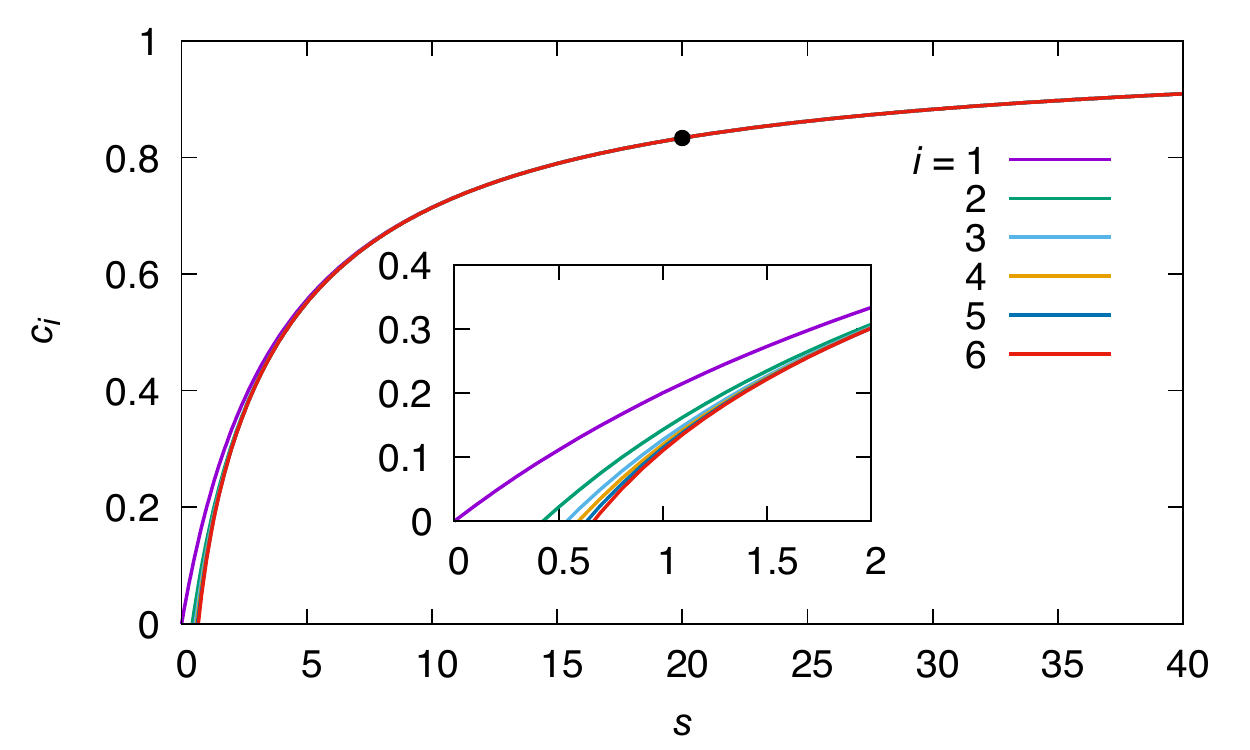}
    \caption{Iterates of equation~(\ref{eq:invar}) for the open Michaelis--Menten reaction mechanisms starting from the initial function $\mathcal{C}_0(s) = 0$ for the parameters of Figure~\ref{fig:traj}(a). The solid dot marks the location of the equilibrium point. The inset shows an expanded view of the behavior of the iterates near origin.
    \label{fig:iterates}}
\end{figure}

As a side note, consider using a vertical initial function, i.e.\ one for which ${d\mathcal{C}/ds =\infty}$. The first iterate from such an initial function is the $s$~nullcline, which
intercepts the $s$~axis at $s=k_0/k_1e_T$, i.e.\ at the extreme right end of the possible range  of $s$~intercepts of the slow manifold. The sQSSA, on the other hand, is the $c$~nullcline, obtained in one iterative step from the initial function $\mathcal{C}_0(s)=0$, and it intercepts the $s$ axis at $s=0$. The two nullclines thus arise naturally as approximations of the slow manifold by iteration from coordinate axes, and serve as upper and lower bounds for the slow manifold. Similar comments about the relationship of the functional equation to the nullclines have previously been made about closed systems~\cite{Fraser1988,FR94,Nguyen1989}.

\subsection{The TFPVs $(k_0,e_T)$ and the small parameters revisited}
If we obtain higher iterates using a symbolic algebra system, then make the substitution~(\ref{ray}), and finally expand in powers of $\varepsilon$, we find that the $i$'th iterate is consistent with the previous iterate to order $\varepsilon^{i-1}$. In other words, the iterative method builds the perturbation series term-by-term, as was previously observed for various perturbative solutions of the closed system~\cite{KK2002,Roussel:1990:GSSA}. However, this property does \textbf{not} hold if we, for instance, expand in powers of $e_T$, since $e_T$ is not, of itself, a TFPV for the open system. These properties parallel those of the direct perturbation calculations.

The first two non-zero terms of the perturbation series computed along the ray~(\ref{ray}) can be written as follows:
\begin{equation}
    \frac{\mathcal{C}(s)}{e_T^*} = \frac{s}{s+K_M}\varepsilon + \frac{K_M\left[s(k_2e_T^*-k_0^*)-k_0^*K_M\right]}{k_1(s+K_M)^4}\varepsilon^2+O(\varepsilon^3).
\label{eq:expank0eT}\end{equation}
Division by $e_T^*$, the nominal value of the enzyme concentration, has made this expression dimensionless. Thus, the $\varepsilon^2$ term
represents an error term for the sQSSA. Specifically, the absolute value of the coefficient of $\varepsilon^2$,
\begin{equation}
    \delta(s) = \left|\frac{K_M\left[s(k_2e_T^*-k_0^*)-k_0^*K_M\right]}{k_1(s+K_M)^4}\right|,
\label{eq:delta}\end{equation}
is a dimensionless error parameter such that the error in the sQSSA is small provided this coefficient is small. An elementary calculation shows that $\delta(s)$ has a local maximum of
\begin{equation}
    \delta_m = \frac{27k_2e_T^*\left(1-\frac{k_0^*}{k_2e_T^*}\right)^4}{256k_1K_M^2}
\label{def:deltam}\end{equation}
in $s\in(0,\infty)$ provided $k_2e_T^*>k_0^*$. The global maximum of $\delta(s)$ for $s\ge 0$ is either this local maximum or
\begin{equation}
    \delta(0)=\varepsilon^*/e_T = \frac{k_0^*}{k_1K_M^2},
\end{equation} 
where the dimensional parameter $\varepsilon^*$ is defined in equation~(\ref{eq:epsilonstar}). %Thus,
%\begin{equation}
%    \delta(s)\le \max\left\{\delta_m H(k_2e_T-k_0),\delta(0)\right\},
%\end{equation}
%where $H(\cdot)$ is the Heaviside function.
When the inflow exceeds the enzyme's clearance capacity, the situation is straightforward, and $\delta(0)$ is the correct small parameter.
Otherwise, we need to establish the parameter conditions under which one or the other of the values of $\delta$ is maximal. Thus, $\delta_m$ will be larger than $\delta(0)$ when
\begin{equation}
    \frac{27}{256}\left(1-\frac{k_0}{k_2e_T}\right)^4 > \frac{k_0}{k_2e_T}.
\end{equation}
We dropped the asterisks here because $k_0^*/k_2e_T^* = k_0/k_2e_T$.
This inequality can be solved numerically. It yields $k_0/k_2e_T<0.0767$. Putting it all together, we have the following:
\begin{itemize}
    \item The sQSSA is a good approximation to the slow manifold globally if $k_0/k_2e_T<0.0767$ and
        $\delta_m\ll 1$. Comparing equation~(\ref{def:deltam}) to~(\ref{epsilon}), and noting that in this parameter range, $\delta_m<\frac{27}{256}\varepsilon_o$, we conclude that $\varepsilon_o\ll 10$ is sufficient for the validity of the sQSSA. This is a somewhat more permissive bound than~(\ref{epsilon}).
    \item If $k_0/k_2e_T>0.0767$, then $\delta(0)\ll 1$ is the appropriate condition for the validity of the sQSSA in the open system.
\end{itemize}
Note that this analysis has recovered both of the small parameters identified in Section~\ref{directEST}, but has also established a sharp boundary for switching from one small parameter to the other. We thus have two complementary methods to obtain small parameters. In any given problem, one or the other method might be unworkable, thus our presentation of both methods here.

\subsection{The TFPVs $(k_0,k_2)$ and the equilibrium approximation}\label{sec:k0k2manif}
We can also expand the iterates using the small parameter implied by~(\ref{Cmank0k2}). If we take $(k_0,k_2)\mapsto\varepsilon(k_0^*,k_2^*)$, and then expand the second (or higher) iterate in powers of $\varepsilon$, we get
\begin{equation}
    \frac{\mathcal{C}(s)}{e_T} = \frac{s}{s+K_E}-\frac{K_E\left(k_2^*s+k_0^*\right)+k_2^*s^2}{k_1\left(s+K_E\right)\left[(s+K_E)^2+K_Ee_T\right]}\varepsilon + O(\varepsilon^2),
\label{eq:expank0k2}\end{equation}
where $K_E=k_{-1}/k_1$. Note that the $O(\varepsilon^0)$ term is the classical quasi-equilibrium approximation (QEA) for the Michaelis--Menten reaction mechanism. Contrast equations (\ref{eq:expank0eT}) and~(\ref{eq:expank0k2}): The QEA for the open system is only accurate to order $\varepsilon^0$, unlike the sQSSA which is accurate to order $\varepsilon$ tThis fact is also reflected in Remark \ref{qssprem}(b) and the second example in Appendix \ref{AppendA}).  This is easily understood given that the QEA lies above the sQSSA at any $s>0$, and that the slow manifold, which enters the first quadrant by passing through the positive $s$ semi-axis, lies below the sQSSA for $s<\widehat{s}$. In the interval $s\in[0,\widehat{s}]$, the sQSSA will therefore always be closer to the slow manifold than the QEA. This is  unlike the situation in the closed system, where the slow manifold lies between the QEA and sQSSA, and where it is possible to choose parameters such that one or the other approximation is more accurate near the origin. The difference is that the QEA is a nullcline in the closed system, but not in the open system. One implication of this result is that the TFPV~(\ref{Cmank0eT}) is the most natural one to use as a basis for a geometric singular perturbation treatment of the slow manifold (see also Appendix~\ref{AppendA} for further notes on the expansion from the TFPV~(\ref{Cmank0k2})).

\subsection{The TFPVs $(k_0,k_1)$ and the linear regime}
Finally, turning to  the TFPV (\ref{Cmank0k1}), we define a perturbation parameter $\varepsilon$ by
\begin{equation}
  (k_0,k_1)\mapsto\varepsilon(k_0^*,k_1^*).  
\label{eq:k0k1eps}\end{equation}
A perturbation series based on this small parameter is a polynomial in $s$ due to the appearance of $k_1$ and~$s$ together in the rate equations. The first nonzero terms of this series are
\begin{equation}
    \frac{\mathcal{C}(s)}{e_T} = \frac{k_1^*s}{k_{-1}+k_2}\varepsilon - \frac{k_1^*\left\{k_1^*s\left[s(k_{-1}+k_2)-k_2e_T\right]+k_0^*(k_{-1}+k_2)\right\}}{(k_{-1}+k_2)^3}\varepsilon^2 + O(\varepsilon^3).
\end{equation}
Substituting this series along with the parameter definitions~(\ref{eq:k0k1eps}) into $\dot{s}$ from~(\ref{mmo}), we get, to lowest order in $\varepsilon$,
\begin{equation}
    \dot{s} \approx \left(k_0^* - \frac{v_\mathrm{max}s}{K_M^*}\right)\varepsilon,
\end{equation}
where $v_\mathrm{max}=k_2e_T$ and $K_M^*=(k_{-1}+k_2)/k_1^*$ or, restoring the small parameters from~(\ref{eq:k0k1eps}),
\begin{equation}
    \dot{s} \approx k_0 - \frac{v_\mathrm{max}s}{K_M}.
\end{equation}
This is of course the small-$s$ linear limit of the sQSSA, the previously seen equation~(\ref{linlimit1}). An alternative route to this equation is presented in Appendix~\ref{AppendA}.
%%%%%%%%%%%%%%%%%%%%%%%%%%%%%%%%%%%%%%%%%%%%%%%%%%%%%%%%%%%%%%

\section{The open Michaelis--Menten reaction mechanism on the Poincar\'e sphere}\label{sec:poin} From a general perspective, it seems worthwhile to consider the global behavior of system~\eqref{mmo} and its distinguished invariant sets to illuminate the role of QSS varieties in a broader context. Proceeding in this manner seems particularly appropriate for systems which do not admit a stationary point in the first quadrant.

It is a standard technique to extend planar polynomial ODE systems to the Poincar\'e sphere. A good description of the procedure is given in Perko \cite{Perko2001}, Section 3.10: Given a sphere in $\mathbb R^3$, let the phase plane be tangent to its north pole, and consider the bijective central projection from the upper half sphere to the phase plane. Then points on the equator of the sphere may be viewed as points at infinity for the planar system, with each line through the origin corresponding to a pair of antipodal points on the equator (the central projection also yields a bijection from the lower hemisphere to the phase plane, and one thus obtains a vector field on the sphere which is mirror symmetric relative to the equatorial plane, and has the equator as an invariant set. One could furthermore pass to a direction field on the projective plane, but we will not do so). Finally, for the purpose of visualization one applies a parallel projection in the north-south direction from the upper hemisphere to the equatorial plane. 

A discussion of the system on the Poincar\'e sphere thus allows us to understand the behavior of the planar system at infinity. Note that all solutions of the system on the Poincar\'e sphere, which is compact, exist for all $t\in\mathbb R$, while this is not necessarily the case for solutions of \eqref{mmo} when $t\leq 0$ or outside the first quadrant. Any reference to limit sets in the following arguments is to be understood for the system on the sphere. In our analysis we will mostly be interested in the first quadrant.

Stationary points at infinity (i.e.\ on the equator) for a polynomial planar system are of particular interest. Antipodal pairs of stationary points generally correspond to invariant lines for the homogeneous part of highest degree; see e.g.~\cite{Walcher2000}. For system \eqref{mmo} with $k_1\not=0$ we thus need to consider the homogeneous quadratic part
\[
\begin{array}{rcl}
\dot s&=& k_1cs,\\
\dot c&=& -k_1cs.
\end{array}
\]
This homogeneous vector field admits three invariant lines, viz.
\[
\mathbb R\cdot\begin{bmatrix} 1\\ 0\end{bmatrix},\quad \mathbb R\cdot\begin{bmatrix} 0\\ 1\end{bmatrix},\quad \mathbb R\cdot\begin{bmatrix} 1\\ -1\end{bmatrix}.
\]
The stationary points at infinity which are relevant for the first quadrant correspond to the rays
\[
\mathbb R_+\cdot\begin{bmatrix} 1\\ 0\end{bmatrix} \text{  and  } \mathbb R_+\cdot\begin{bmatrix} 0\\ 1\end{bmatrix},
\]
and we call the corresponding stationary points at infinity $P_1$, resp. $P_2$. Moreover we denote by $P_3$ the stationary point at infinity which corresponds to $\mathbb R_+\cdot\begin{bmatrix} 1\\ -1\end{bmatrix} $. 

We present the pertinent results for system \eqref{mmo} on the Poincar\'e sphere (see Appendix \ref{poincapp} for computations and proofs).
\begin{lemma}\label{ptlem0}Assume that the genericity conditions \eqref{genercond} are satisfied. Then the following hold for the system on the Poincar\'e sphere.
\begin{enumerate}[(a)]
\item The stationary point $P_1$ at infinity is a degenerate saddle when $k_2e_T>k_0$, with the stable manifold contained in the equator. In case $k_2e_T<k_0$ this point is a degenerate attracting node.
\item The stationary point $P_2$ at infinity is a saddle-node, with a repelling node part on the upper hemisphere.
\item The stationary point $P_3$ at infinity is a repelling node.
\end{enumerate}
\end{lemma}

We first describe the behavior of system \eqref{mmo} on the relevant part of the  Poincar\'e sphere when there is an isolated stationary point in the first quadrant; see also Figure \ref{fig:PoinA}.
\begin{proposition}\label{psprop1}
Assume that the genericity conditions \eqref{genercond} hold, and let $k_2e_T>k_0$. Then every solution starting in the first quadrant converges toward $P_0=(\widehat s,\,\widehat c)$ as $t\to \infty$. There is a unique distinguished trajectory that connects the saddle $P_1$ at infinity to $P_0$. Moreover this trajectory is asymptotic in the phase plane to the line $c=e_T$ as $t\to-\infty$.
\end{proposition}
\begin{figure}
    \centering
    \includegraphics[scale=0.9]{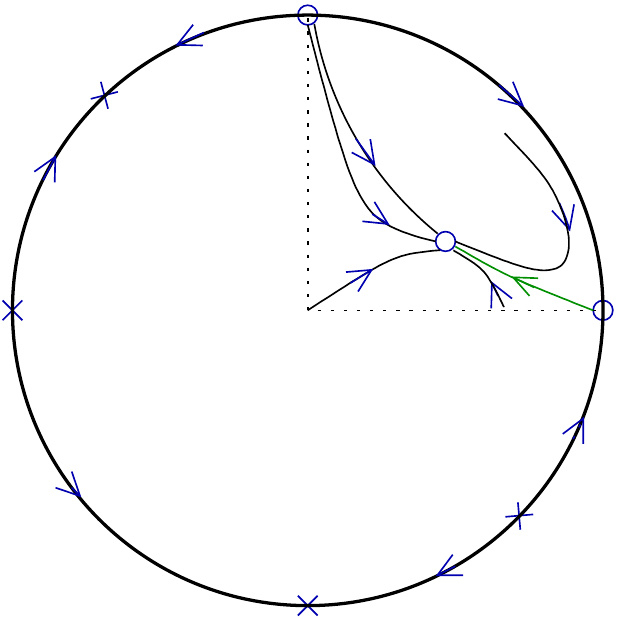}
    \caption{The system on the Poincar\'{e} Sphere for the open Michaelis--Menten reaction mechanism in case $k_2e_T>k_0$. The distinguished trajectory is colored green.
    \label{fig:PoinA}}
\end{figure}

We turn to the case when $P_0$ lies in the second quadrant; see Figure \ref{fig:PoinB}. Here, considering the system on the Poincar\'e sphere is necessary to understand the global dynamics, and moreover a proper understanding requires us to look beyond the first quadrant.
\begin{proposition}\label{psprop2}
Assume that the genericity conditions \eqref{genercond} hold, and let $k_2e_T<k_0$. Then every solution that starts in the first quadrant converges to $P_1$ as $t\to\infty$, and its trajectory in the phase plane is asymptotic to the line $c=e_T$. There is a unique distinguished trajectory which connects the saddle $P_0$ to $P_1$.
\end{proposition}
\begin{figure}
    \centering
    \includegraphics[scale=0.9]{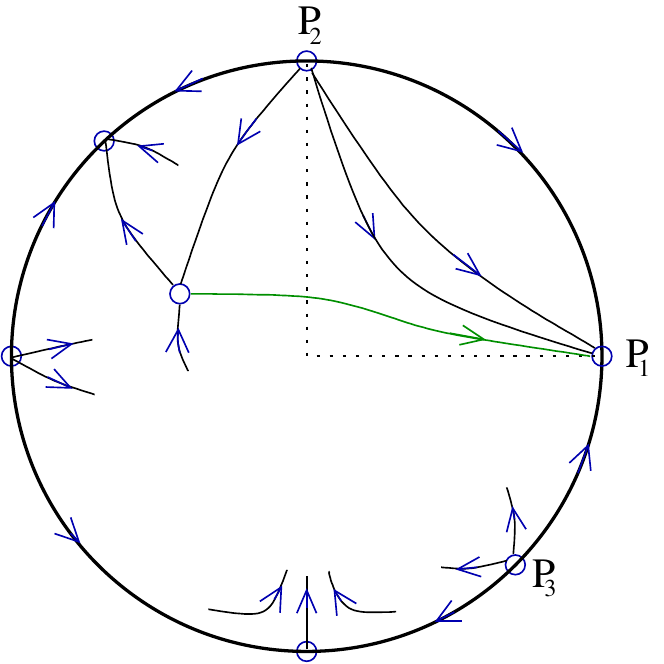}
    \caption{The system on the Poincar\'{e} Sphere for the open Michaelis--Menten reaction mechanism in case $k_2e_T<k_0$. The distinguished trajectory is colored green.
    \label{fig:PoinB}}
\end{figure}
\begin{remark}
In view of Proposition \ref{psprop2}, the mathematically distinguished trajectory connecting $P_0$ and $P_1$ may be seen as a natural candidate for a ``global''  slow manifold in appropriate parameter regimes. We provide a few more details here. From the proof (see also Figure \ref{fig:PoinB}) one finds that the two components of the unstable manifold of $P_0$ connect to $P_1$, resp.\ to the antipode of $P_3$. Solutions in the open upper hemisphere, unless they start on the stable manifold of $P_0$, converge either to $P_1$ or to the antipode of $P_3$ as $t\to\infty$. Moreover one component of the stable manifold of $P_0$ connects to $P_2$ (which is the only available alpha limit point), and the other may connect either to the antipode of $P_1$, or to the antipode of $P_2$, or to $P_3$ (topological arguments do not yield more precise information, and here we will not delve any further into this matter). In any case, the stable manifold of $P_0$ separates the regions of attraction for $P_1$ and the antipode of $P_3$ in the open upper hemisphere. In turn, the region of attraction for $P_1$ is separated by the distinguished trajectory into two subregions. For one of these subregions, the alpha limit set of all points is equal to $\{P_2\}$, thus one may briefly say that all trajectories in this region come down from $c=\infty$. For the other subregion, a similarly concise statement does not seem possible: The set of alpha limit points certainly includes $P_3$, but it may also include the antipode of $P_2$ or of $P_1$.
\end{remark}

\section{Discussion}

The open Michaelis--Menten reaction mechanism, although of definitive relevance in biochemistry, has attracted less attention than the classical closed mechanism without influx. We investigated the sQSSA for this system from two perspectives. On the one hand, we considered QSS from a singular perturbation viewpoint, determined all TFPVs from which singular perturbation reductions emanate and identified the relevant parameter values for sQSSA. On the other hand, motivated by the results of Stoleriu et al.~\ \cite{Stoleriu2004}, we started from a less restrictive notion of QSS and obtained sQSSA results by direct estimates for a wider range of parameters (such a phenomenon does not appear in the closed Michaelis--Menten system). By these estimates we obtained a justification of central results in \cite{Stoleriu2004}, and could also extend their range. Considering the fine structure of slow manifolds by analysis of higher order approximations revealed the special role (and higher accuracy of approximation) for parameters that are related to singular perturbations. Finally, we took a global perspective to investigate scenarios with no positive equilibrium.

\appendix
\section{Projecting onto a slow manifold}\label{AppendA}

This appendix is intended to give the reader a short and user-friendly overview of pertinent methods in geometric singular perturbation theory. For a more elaborate presentation, we encourage the reader to consult \cite{Fenichel1971, Fenichel1979,Goeke2014,kuehn2015,Noethen2011,Wechselberger2020}. The references \cite{Goeke2014,Noethen2011} focus specifically on the QSSA.

Singular perturbation reductions are straightforward for sufficiently smooth systems in standard form
\[
\begin{array}{rcl}
\dot x_1&=&\varepsilon f_1(x_1,x_2,\varepsilon),\\
\dot x_2&=&f_2(x_1,x_2,\varepsilon),\\
\end{array}
\]
which depend on a ``small parameter'' $\varepsilon$. We assume that the reader is familiar with this procedure, which we sketch without giving details: In slow time $\tau=\varepsilon t$ the system may be rewritten as
\[
\begin{array}{rcl}
x_1^\prime&=&f_1(x_1,x_2,\varepsilon),\\
\varepsilon x_2^\prime&=&f_2(x_1,x_2,\varepsilon).\\
\end{array}
\]
Given suitable hyperbolicity conditions, solutions of the latter system converge toward solutions of the reduced (differential-algebraic) system 
\[
x_1^\prime=f_1(x_1,x_2,0),\quad f_2(x_1,x_2,0)=0.
\]

The procedure just sketched requires an a priori separation of slow and fast variables, which is not necessarily given. This difficulty is overcome by a coordinate-free version of singular perturbation reductions as developed by Fenichel \cite{Fenichel1979} (slmost all the relevant information is contained in pp.\ 65--66 of this reference, but in rather condensed form). Here we present the basic theory and computation-relevant facts, loosely following Wechselberger \cite{Wechselberger2020}, Chapter 3, as well as \cite{Goeke2013, Goeke2014} specifically for the QSS reduction procedure.

For systems not in standard form, one must first \textit{define} the notion of a singular perturbation, according to Fenichel. Given a differential equation of the form
\begin{equation}\label{GeneralForm}
\dot{x} = F(x,\varepsilon),  \quad x \in \mathbb{R}^n, \quad F:\mathbb{R}^n \times \mathbb{R} \mapsto \mathbb{R}^n
\end{equation}
with sufficiently smooth $F$, one is interested in the dynamics in the asymptotic limit $\varepsilon \to 0+$.
The singular points of $F(x,0)$ determine the nature of the perturbation. It will be convenient to express $F$ in the form
\begin{equation}
  F(x,\varepsilon) =: h(x) + \varepsilon G(x,\varepsilon),
\end{equation}
so that there is a clear distinction between the vector field, $h(x)=F(x,0)$, and the perturbation, $\varepsilon G(x,\varepsilon)$. Let $S$ denote the set of singular points of $F(x,0)$:
\begin{equation}
S:=\{x\in \mathbb{R}^n : h(x)=0\}.
\end{equation}
Note that $S$ is an algebraic variety for polynomial or rational systems, which are quite common in reaction equations. If $S$ is the empty set, or contains only isolated singularities, then the perturbation is called \textit{regular}. In contrast, if $M\subseteq S$ is a differentiable manifold comprised of non-isolated singularities, then the perturbation is singular, and $M$ is called a \textit{critical manifold}. 

We now outline the reduction procedure in the coordinate-free setting:
\begin{enumerate}
\item
Given a singularly perturbed problem the form (\ref{GeneralForm}), one can establish necessary and sufficient conditions for the existence of a local transformation to standard form, as follows: For every $x\in M$ the required conditions are
\begin{enumerate}[(i)]
\item 
$T_xM= \ker Dh(x)=\{v\in \mathbb{R}^n: Dh(x)\cdot v=0\}.
$
\item For the eigenvalue $0$ of $Dh(x)$ the algebraic and the geometric multiplicities are equal.
\item Normal hyperbolicity: All nonzero eigenvalues of $Dh(x)$ have nonzero real parts. In applications one often requires the stronger attracting hyperbolicity condition, viz.\ all nonzero eigenvalues of $Dh(x)$ have negative real parts.
\end{enumerate}
\item Given the above conditions, one has a direct sum decomposition 
\begin{equation}
    \mathbb{R}^n = T_xM \oplus N_x, \quad \forall x \in M,
\end{equation}
where the $Dh(x)$-invariant complementary subspace $N_x={\mathcal R}(Dh(x))$ is the range of $Dh(x)$. The conditions in item 1 are necessary and sufficient for such a decomposition to exist.\\
Given this decomposition one can define the projection operator, $\Pi^M$, which for every $x\in M$ maps $\mathbb{R}^n$ onto the tangent space of $M$ at $x$ and has kernel $N_x$ (recall that a projection is uniquely determined by its kernel and image).\\
Once $\Pi^M$ is known, the leading order singular perturbation reduction is computed by projecting the perturbation term onto the tangent space of $M$ of~$x$:
\begin{equation}\label{QSS}
{x}^\prime = \Pi^M G(x,0),\quad G(x,0)\in \mathbb{R}^n, \quad x \in M.
\end{equation}
\item 
To compute $\Pi^M$ explicitly, it is useful to employ a decomposition
\begin{equation}
    h(x) = P(x)f(x),
\end{equation}
where $P(x)$ is a rectangular matrix valued function, $M$ locally coincides with the zero level set of the vector valued function $f(x)$, and $Df(x)$ has full rank when $x\in M$. The existence of such a decomposition is guaranteed by the implicit function theorem, and for polynomial or rational systems it can be obtained in an algorithmic manner.
One obtains the operator $\Pi^M$ as
\begin{equation}
 \Pi^M := I- P(Df P)^{-1}Df.
\end{equation}
\item In addition to the reduced equation, the initial value on the critical manifold $M$ is also relevant. In the attracting hyperbolic case the fast equation $\dot x =h(x)$ admits $n-\dim M$ independent first integrals in a neighborhood of $M$, and to a given initial value $z\in\mathbb R^n$ corresponds (up to a correction of order $\varepsilon$) the point where $M$ intersects the common level set of the first integrals containing~$z$; see Fenichel \cite{Fenichel1979}, Lemma 5.3, and also \cite{Goeke2014}, Proposition 2.
\end{enumerate}

Formally, the procedure resulting in (\ref{QSS}) is referred to as \textit{slow manifold projection}; see Figure \ref{fig:proj} for a geometric illustration. The flow on $M$ at $\varepsilon=0$ is trivial, but Fenichel theory ensures that the perturbed vector field has an invariant \textit{slow manifold} close to $M$, on which the flow is slow but non-trivial. The long-time evolution of $x$ is given (approximately) by the projected dynamical system \eqref{QSS}.

\begin{figure}
    \centering
    \includegraphics[scale=2.5]{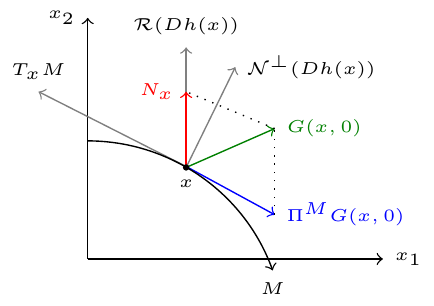}
    \caption{{Projecting onto the slow manifold.} In this figure, ``$\mathcal{R}$" denotes range and ``$\mathcal{N}$" denotes nullspace. The complementary subspaces $T_xM$ and $N_x$ are invariant with respect to the linearization $Dh(x)$, and the components of $G(x,0)\in \mathbb{R}^n$ can be uniquely expressed as $G(x,0)=u+v$, with $u\in T_xM$ and $v\in N_x$. $\Pi^{M}$ is constructed in the form of an oblique projection onto $T_xM$; note that $T_xM$ and $N_x$ are not necessarily orthogonal. The perturbed dynamical system that is influenced by the presence of $G(x,0)$ is approximated by (\ref{QSS}). Note that the critical manifold $M$ is in fact filled with non-isolated equilibria.
    \label{fig:proj}}
\end{figure}

As a first illustrating example we formally compute the singular perturbation reduction for the case of small $k_0=\varepsilon k_0^*$ and small $k_1=\varepsilon k_1^*$, thus we have the perturbation problem
\begin{equation}\label{p2}
\begin{array}{rcl}
\dot s&=& \varepsilon k_0^*-\varepsilon k_1^*( e_T-c)s + k_{-1}c,\\
\dot c&=& \varepsilon k_1^*(e_T-c)s -( k_{-1}+k_2)c.\\
\end{array}
\end{equation}
The critical manifold, $M$, attained by setting $\varepsilon=0$ in (\ref{p2}), corresponds to the $s$~axis:
\begin{equation}
    M:=\{(s,c)\in\mathbb{R}^2: c=0\}.
\end{equation}
Furthermore, $h(s,c)=P(s,c)f(s,c)$ and $G(s,c,\varepsilon)$ are given by
\begin{equation}
   P(s,c) :=\begin{bmatrix}k_{-1}\\
    -(k_{-1}+k_2)\end{bmatrix},\quad f(s,c)=c, \quad G(s,c,\varepsilon):=\begin{bmatrix}k_0^*-k_1^*(e_T-c)s\\
    k_1^*(e_T-c)s\end{bmatrix},
\end{equation}
and thus $Df=[0 \;\; 1]$. Next, $DfP$ is the scalar $-(k_{-1}+k_2)$. The product of $P$ and $Df$ is
\begin{equation}
    P Df := \begin{bmatrix} 0 & k_{-1} \\ 0 & -(k_{-1}+k_2) \end{bmatrix}.
\end{equation}
Combining the above results yields
\begin{equation}
    \Pi^M=I_{2 \times 2}+\cfrac{1}{(k_{-1}+k_2)} P Df =\begin{bmatrix} 1 & \cfrac{k_{-1}}{k_{-1}+k_2} \\ 0 & 0 \end{bmatrix},
\end{equation}
and the reduced equation is
\begin{equation}
    s' := \Pi^M|_{c=0} G(s,0,0)= k_0^*-\cfrac{k_1^*k_2e_T}{k_{-1}+k_2}s.
\label{eq:sprimek0k1}\end{equation}

It is worth pointing out that the fast system here admits the first integral: $(k_{-1}+k_2)s+k_{-1}c$. Hence, given an initial value $(s_0,c_0)$ for \eqref{p2}, the corresponding initial value for the reduced equation is just 
\[
\widetilde s_0=s_0+\frac{k_{-1}}{k_{-1}+k_2}c_0.
\]
In this example the singular perturbation reduction~(\ref{eq:sprimek0k1}) coincides with the ``classical'' QSS reduction with respect to~$c$ in the linear regime where $s\ll K_M$. This is not accidental, but due to the special form of the critical manifold; see \cite{Goeke2017}, Proposition~5.

As a second example consider the singular perturbation reduction for the case of small $k_0=\varepsilon k_0^*$ and small $k_2=\varepsilon k_2^*$, thus yielding the perturbation problem
\begin{equation}\label{p3}
\begin{array}{rcl}
\dot s&=& \varepsilon k_0^*- k_1( e_T-c)s + k_{-1}c,\\
\dot c&=& k_1(e_T-c)s -( k_{-1}+ \varepsilon k_2^*)c.\\
\end{array}
\end{equation}
Here the ``classical'' QSS reduction is significantly different from the singular perturbation reduction. The critical manifold is defined by
\begin{equation}
    M:=\{(s,c)\in\mathbb{R}^2: f(s,c)=0\}
\end{equation}
with $f(s,c):=k_1(e_T-c)s-k_{-1}c$, and
furthermore, $h(s,c)=P(s,c)f(s,c)$ with
\begin{equation}
   P(s,c) :=\begin{bmatrix}-1\\
    1\end{bmatrix}; \quad \text{moreover  }G(s,c,\varepsilon):=\begin{bmatrix}k_0^*\\
    -k_2^*c\end{bmatrix}.
\end{equation}
A routine calculation yields the reduced system
\begin{equation}
\begin{bmatrix} s^\prime\\
c^\prime\end{bmatrix} =\frac{k_0^*-k_2^*c}{k_1(e_T-c)+k_1s+k_{-1}}\begin{bmatrix}k_1s+k_{-1}\\ k_1(e_T-c)\end{bmatrix}
\end{equation}
which is relevant only on the invariant manifold $M$. Using the parameterization $c=k_1e_Ts/(k_1s+k_{-1})$ of $M$, one arrives at an equation for $s$ alone:
\begin{equation}\label{nonqss}
s^\prime=(k_1s+k_{-1})\cdot\frac{k_0^*(k_1s+k_{-1})-k_2^*k_1e_Ts}{k_1k_{-1}e_T+(k_1s+k_{-1})^2}.
\end{equation}
The reduced system without inflow is known from the literature, see e.g.\ \cite[Example 8.6]{Goeke2013}, \cite[Section 5]{Roussel2019b} or \cite[Section 3.4]{Wechselberger2020}. Note that $s+c$ is a first integral of the fast system; this may be employed to determine the initial value on $M$.
%%%%%%%%%%%%%%%%%%%%%%%%%%%%%%%%%%%%%%%%%%%%%%%%%%

\section{Computations and proofs for the Poincar\'e sphere:  }\label{poincapp} In this Appendix, we record the necessary computations, and give proofs for Lemma \ref{ptlem0} as well as Propositions \ref{psprop1} and \ref{psprop2}.

From a computational perspective it is convenient to project the system from the sphere to another tangent plane. The procedure was streamlined (for different purposes) in \cite{Walcher2000}, and we are using it here, noting that the final result is the same as in \cite{Perko2001}:

To accommodate various transformations we rename the variables in system \eqref{mmo}, thus obtaining
\begin{equation}\label{mmox}
\begin{array}{rcl}
\dot x_1&=& k_0-k_1(e_T-x_2)x_1 + k_{-1}x_2,\\
\dot x_2&=& k_1(e_T-x_2)x_1 -( k_{-1}+k_2)x_2.
\end{array}
\end{equation}
We compute the Poincar\'e transform of this system with respect to $x_1$ (terminology from \cite{Walcher2000}), which corresponds to the transformed system on the tangent plane to the ``east pole'' (compare Perko \cite{Perko2001}, Section 3.10, Theorem 2).
\begin{itemize}
\item In a first step introduce a further variable $x_3$ and homogenize, to obtain
\[
\begin{array}{rclcl}
\dot x_1&=& k_0x_3^2-k_1e_Tx_1x_3+k_1x_1 x_2 + k_{-1}x_2x_3 &=:&g_1,\\
\dot x_2&=& k_1e_Tx_1x_3-k_1x_1x_2 -( k_{-1}+k_2)x_2x_3 &=:&g_2.
\end{array}
\]
\item In step 2, compute the projected system
\[
\begin{array}{rcl}
\dot x_2&=& -x_2g_1+x_1g_2,\\
\dot x_3&=&-x_3g_1.
\end{array}
\]
\item In step 3, dehomogenize by setting $x_1=1$, to obtain
\begin{equation}\label{mmopt1}
\begin{array}{rcl}
\dot x_2&=& -k_1x_2+k_1e_Tx_3-k_1x_2^2+(k_1e_T-k_{-1}-k_2)x_2x_3-k_{-1}x_2^2x_3-k_0x_2x_3^2,\\
\dot x_3&=& -x_3\left(k_1x_2-k_1e_Tx_3+k_{-1}x_2x_3+k_0x_3^2\right).
\end{array}
\end{equation}
with the equator corresponding to $x_3=0$.
\end{itemize}

\begin{lemma}\label{ptlem1}Assume that the genericity conditions \eqref{genercond} are satisfied. Then the following hold:
\begin{enumerate}[(a)]
\item System \eqref{mmopt1} admits two stationary points on the line $x_3=0$. These are $(0,\,0)$, corresponding to $P_1$ for the homogeneous quadratic part of system \eqref{mmox}, and $(-1,\,0)$ which corresponds to  $P_3$.
\item At the stationary point $(0,\,0)$ the Jacobian is 
\[
\begin{bmatrix} -k_1 & k_1e_T\\ 0&0\end{bmatrix}.
\]
In case $k_2e_T>k_0$ this point is a degenerate saddle with the equator as local stable manifold and a center-unstable manifold tangent to the line $x_2-e_Tx_3=0$. In case $k_2e_T<k_0$ this point is a degenerate attracting node with all trajectories but the two on the equator approaching it tangent to the line $x_2-e_Tx_3=0$.
\item At the stationary point $(-1,\,0)$ (which is irrelevant for the dynamics on the first quadrant) the Jacobian is 
\[
\begin{bmatrix} k_1 & k_2\\ 0&k_1\end{bmatrix},
\]
hence this point is a repelling node.
\end{enumerate}
\end{lemma}
\begin{proof}
\begin{enumerate}[(i)]
    \item The stationary points with $x_3=0$ are determined from the equation $-k_1(x_2+x_2^2)=0$; thus part (a) follows. Computing the Jacobians is straightforward, and part (c) as well as the first statement of (b) follows. To prepare for proving the remaining statements, introduce new coordinates $y_2= x_2-e_Tx_3$ and $y_3=x_3$ to obtain
\[
\begin{array}{rcccl}
\dot y_2&=&-k_1y_2&-&(k_{-1}+k_2)e_Ty_3^2+\cdots \\
\dot y_3&=&           &-& k_1y_2y_3-(k_{-1}e_T+k_0)y_3^3+\cdots
\end{array}
\]
with diagonalized linear part.
\item 
The following auxiliary result is a special case from the last section of \cite{Walcher1993}:
Consider a system
\begin{equation}\label{nfim2d}
\begin{array}{rcccccl}
\dot z_1&=&  & &\alpha_{111}z_1^2+2 \alpha_{112}z_1z_2+\alpha_{122}z_2^2&+&\gamma_{1111}z_1^3+\cdots \\
\dot z_2&=& \lambda z_2&+& \alpha_{211}z_1^2+2 \alpha_{212}z_1z_2+\alpha_{222}z_2^2&+&\cdots
\end{array}
\end{equation}
with real parameters, and $\lambda\not=0$. Then the normal form on an invariant manifold (NFIM) tangent to $z_2=0$, up to degree three, is given by
\begin{equation}\label{nfim1d}
\dot z_1=\alpha_{111}z_1^2+\left(\gamma_{1111}-2\alpha_{112}\alpha_{211}/\lambda\right)z_1^3+\cdots
\end{equation}
In case $\alpha_{111}\not=0$ the stationary point $0$ of system \eqref{nfim2d} is a saddle-node. In case $\alpha_{111}=0$ and $\lambda<0$ the stationary point $0$ is a degenerate saddle when $\gamma_{1111}-2\alpha_{112}\alpha_{211}/\lambda>0$ and a degenerate attracting node when $\gamma_{1111}-2\alpha_{112}\alpha_{211}/\lambda<0$. In the latter case all but two trajectories are tangent to $z_2=0$.
\item Applying this result with the identifications $z_1=y_3$, $z_2=y_2$, $\alpha_{111}=0$, $2\alpha_{112}=-k_1$, $\gamma_{1111}=-(k_{-1}e_T+k_0)$ and $\alpha_{211}=-(k_{-1}+k_2)e_T$ one obtains the NFIM up to degree three as
\[
\dot y_3=(k_2e_T-k_0)y_3^3+\cdots,
\]
and all assertions follow.
\end{enumerate}

\end{proof}
There remains to discuss the stationary point $P_2$ at infinity, for which we use the Poincar\'e transform of \eqref{mmox} with respect to $x_2$. The first step is unchanged but the projection in the second step is now given as
\[
\begin{array}{rcl}
\dot x_1&=&-x_1g_2+x_2g_1,\\
\dot x_3&=& -x_3 g_2
\end{array}
\]
and dehomogenization $x_2=1$ yields 
\begin{equation}\label{mmopt2}
\begin{array}{rcl}
\dot x_1&=& -k_1(x_1+x_1^2)+k_{-1}x_3+(k_{-1}+k_2+k_1e_T)x_1x_3+k_0x_3^2-k_1e_Tx_1^2x_3,\\
\dot x_3&=& -x_3\left(k_1e_Tx_1x_3-k_1x_1 -(k_{-1}+k_2)x_3\right).
\end{array}
\end{equation}
There are two stationary points with $x_3=0$. The point $(-1,0)$ corresponds to $P_3$, which has been taken care of. There remains $(0,0)$, corresponding to $P_2$.
\begin{lemma}\label{ptlem2} 
Assume that the genericity conditions \eqref{genercond} hold. Then the stationary point $P_2$ at infinity is a saddle-node, with a repelling node part on the upper hemisphere. Except for the trajectories on the equator, all trajectories of \eqref{mmopt2} with positive $x_3$ that emanate from this stationary point are tangent to the line given by $k_1x_1+k_{-1}x_3=0$.
\end{lemma}
\begin{proof} To diagonalize the Jacobian at $(0,0)$, introduce new coordinates $y_1=x_1+k_{-1}/k_1\cdot x_3$, $y_3=x_3$ to obtain the system
\[
\begin{array}{rcl}
\dot y_1&=& k_1y_1+\cdots\\
\dot y_3&=& k_2y_3^2+k_1y_1y_3+\cdots
\end{array}
\]
with the dots denoting terms of higher order. By the result quoted in the proof of Lemma \ref{ptlem1}, the NFIM up to degree two on $y_1=0$ is given by $\dot y_3=k_2y_3^2+\cdots$, and all assertions follow.
\end{proof}
We note that Lemma \ref{ptlem0} is thus proven. We turn to the Propositions.

\begin{proof}[Proof of Proposition \ref{psprop1}]
By Poincar\'e-Bendixson and Lemma \ref{basicfacts}, the omega limit set of every solution starting  in the first quadrant must contain a stationary point. By Lemma \ref{ptlem2}, $P_2$ is a saddle-node with a repelling saddle part in the upper hemisphere, so $P_2$ cannot be an omega limit point. By Lemma \ref{ptlem1} the point $P_1$ is a saddle with stable manifold on the equator, therefore an omega limit set containing $P_1$ cannot consist of $P_1$ alone. By the Butler-McGehee theorem (see e.g.\ Smith and Waltman \cite{SmithWaltman1995}), the omega limit set has nonempty intersection with the stable manifold of the saddle, and by invariance and closedness it must contain $P_2$ or $P_3$; a contradiction since both these points are repelling. So, only $P_0$ remains, and by attractivity of $P_0$ the solution converges toward this point. The last two assertions are concerned with the center-unstable manifold of $P_1$, and are a consequence of Lemma \ref{ptlem1}(b), since dehomogenizing $x_2-e_Tx_3$ with respect to $x_3$ yields $x_2-e_T$.
\end{proof}

\begin{proof}[Proof of Proposition \ref{psprop2}] By Lemmas \ref{ptlem1} and \ref{ptlem2}, and by properties of antipodal points for systems of even degree, the stationary points at infinity are $P_1$ (attracting node), its antipode (repelling node), $P_2$ (repelling node part for upper hemisphere), its antipode (saddle part for the upper hemisphere, with stable manifold on the equator), $P_3$ (repelling node) and its antipode (attracting node).

The first two statements follow from Lemma \ref{ptlem1} and Poincar\'e-Bendixson theory, similar to the previous case. For the last statement, consider the two local components of the unstable manifold of $P_0$. Such a component cannot connect to a component of the stable manifold, since the existence of a homoclinic orbit would imply the existence of a further stationary point. Therefore the omega limit set of a point on the unstable manifold must contain a different stationary point. This point cannot be a repelling node, which excludes $P_3$ or the antipode of $P_1$, and it cannot be $P_2$, which is repelling for the upper hemisphere. Finally, it cannot be the antipode of $P_2$ with saddle part in the upper hemisphere, by arguments analogous to those in the previous proof, invoking the Butler-McGehee theorem. Therefore, the omega limit set of a point in the unstable manifold of $P_0$ contains a single point, which is either $P_1$ or the antipode of $P_3$. Finally, not both local components of the unstable manifold can have the same point as omega limit point; this again would imply the existence of a further stationary point.

\end{proof}
\begin{remark}
In case $k_2e_T=k_0$, system \eqref{mmo} admits no finite stationary point, and for the sake of completeness we record the pertinent result about this setting. By routine (albeit lengthy) computations one finds that
in case $k_2e_T=k_0$, $k_1\not=0$ and $k_2\not=0$ the NFIM of \eqref{mmopt1} at $0$ up to degree four is given by
\[
\dot y_3=-\frac{(k_{-1}+k_2)k_2e_T}{k_1}y_3^4+\cdots,
\]
hence the stationary point $P_1$ is a degenerate attracting node. In the global picture, $P_1$ thus attracts all solutions starting in the first quadrant.
\end{remark}

%For acknowledgements section, please don't number the section, please begin it with \section*{Acknowledgements}
% \section*{Acknowledgments} We would like to thank you for \textbf{following
% the instructions above} very closely in advance. It will definitely
% save us lot of time and expedite the process of your paper's
% publication.

% You may incorporate your references as follows in your main tex file.
% Using BibTex is not recommended but can be handled.

% \bibliographystyle{AIMS}
% \bibliography{Bibliography.bib}

\providecommand{\href}[2]{#2}
\providecommand{\arxiv}[1]{\href{http://arxiv.org/abs/#1}{arXiv:#1}}
\providecommand{\url}[1]{\texttt{#1}}
\providecommand{\urlprefix}{URL }

\end{document}